\documentclass[10pt,reqno]{amsart}

\usepackage[dvipsnames]{xcolor}
\usepackage{amsmath,amssymb,amsthm,amsfonts,enumerate,tikz,bm}
\usepackage{mathrsfs}
\usetikzlibrary{matrix,arrows,positioning,calc}
\usepackage[all]{xy}
\usepackage[bookmarksnumbered,colorlinks]{hyperref}
\usepackage{url}
\numberwithin{equation}{section}
\usepackage{graphicx}
\usepackage[T1]{fontenc}
\usepackage{textcomp}
\usepackage{palatino,helvet}

\usepackage{footmisc}
\usepackage{rotating}
\usepackage{tikz-cd}
\usepackage{adjustbox}
\usepackage{scrextend}
\usepackage{orcidlink}

\newtheorem{definition}{Definition}[section]
\newtheorem{remark}[definition]{Remark}
\newtheorem{example}[definition]{Example}

\newtheorem{theorem}[definition]{Theorem}
\newtheorem{proposition}[definition]{Proposition}
\newtheorem{lemma}[definition]{Lemma}
\newtheorem{corollary}[definition]{Corollary}
\newtheorem{notation}[definition]{Notation}
\newtheorem{teo}{Theorem}
\theoremstyle{remark}

\usepackage{paralist}

\usepackage{scrextend}
\deffootnote{2em}{0em}{\thefootnotemark\quad}

\newcommand{\mbE}{\mathbb{E}}
\newcommand{\mcA}{\mathcal{A}}
\newcommand{\mcB}{\mathcal{B}}
\newcommand{\mcC}{\mathcal{C}}
\newcommand{\mcD}{\mathcal{D}}

\newcommand{\mcI}{\mathcal{I}}

\newcommand{\mcP}{\mathcal{P}}
\newcommand{\mcS}{\mathcal{S}}
\newcommand{\mcX}{\mathcal{X}}
\newcommand{\mcY}{\mathcal{Y}}

\newcommand{\modu}{\mathrm{mod}}
\newcommand{\add}{\mathrm{add}}
\newcommand{\free}{\mathrm{free}}
\newcommand{\smd}{\mathrm{smd}}

\newcommand{\ind}{\mathrm{ind}}

\newcommand{\fgMod}{\mathsf{mod}}

\newcommand{\Hom}{\mathrm{Hom}}
\newcommand{\Ext}{\mathrm{Ext}}
\newcommand{\Obj}{\mathrm{Obj}}
\newcommand{\Mor}{\mathrm{Mor}}

\newcommand{\Ker}{\mathrm{Ker}}

\makeatletter
\def\@seccntformat#1{%
  \protect\textup{\protect\@secnumfont
    \ifnum\pdfstrcmp{section}{#1}=0 \scshape\bfseries\fi% section # in \scshape and \bfseries
    \ifnum\pdfstrcmp{subsection}{#1}=0 \bfseries\fi% subsection # in \bfseries
    \csname the#1\endcsname
    \protect\@secnumpunct
  }%
}
\makeatother

\begin{document}

\title{On the Iyama-Yoshino reduction in extriangulated categories}
\thanks{2020 MSC: 18G80 (18G10; 18G25)}
\thanks{Key Words: Iyama-Yoshino reduction, $(n+2)$-rigid subcategories, extriangulated categories}

\author{Mindy Y. Huerta} % \orcidlink{0000-0002-6811-5686}
\dedicatory{Dedicated to professor Corina Sáenz on the occasion of her 60th birthday}
\address[M. Y. Huerta]{Facultad de Ciencias. Universidad Nacional Aut\'onoma de M\'exico. Circuito Exterior, Ciudad Universitaria. CP04510. Mexico City, MEXICO}
\email{mindyhp90@ciencias.unam.mx}

\begin{abstract}
In this paper, we provide an interpretation of the existing reduction process for extriangulated categories in general. This process allows us to obtain a new category which, for well-known cases, admits a triangulated structure. We will show that in general this new category is extriangulated but not necessarily triangulated.
\end{abstract}

\maketitle

%\setcounter{tocdepth}{2}
%\tableofcontents

\pagestyle{myheadings}
\markboth{\rightline {\scriptsize M. Y. Huerta}}
         {\leftline{\scriptsize On the Iyama-Yoshino reduction in extriangulated categories}}

\section*{\textbf{Introduction}}

Iyama-Yoshino reduction firstly appeared in \cite{IYmutation} describes
a process to get a $2$-Calabi-Yau triangulated category from another one through rigid objects. Among the two most important applications we have: (1) a bijection
between its cluster-tilting subcategories and those in the original category having 
as element the rigid object, and (2) Iyama-Yoshino reduction is closely 
related to other reduction techniques in representation theory (see \cite{jasso2015reduction, aihara2012silting}, for instance).

On the other hand, some years ago the notion of extriangulated categories was introduced in \cite{NP19} by 
H. Nakaoka and Y. Palu. This notion generalizes at 
the same time triangulated and exact categories and 
its importance lies in the fact that several works have been carried to this new context providing
new outcomes and proofs. Just to mention an example,
recently J. C. Cala and S. R. Hern\'andez proved a
characterization of closed subfuntors through $3\times 3$-lemma property in extriangulated categories where the proofs done in other contexts 
have no place (see \cite{CH25}). Due to the enormous usefulness that extriangulated categories have to generalize new theory, many topics can be studied from a
new point of view. As we can imagine, it has not been long time for the reduction process to be addressed in extriangulated categories.

In \cite{FMP23}, E. Faber, B. R. Marsh and M. Pressland provide a reduction technique for Frobenius extriangulated categories 
generalizing Iyama-Yoshino reduction which is
recovered by passing to stable categories. Both reductions are triangulated categories and, indeed,
there is a triangle equivalence between them.

\begin{teo}\cite[Theorem 4.16]{FMP23}\label{teo 1} If $\mathcal{F}$ is a stably $2$-Calabi-Yau Frobenius extriangulated 
category and $\mcX\subseteq \mathcal{F}$ is a funtorially finite rigid subcategory, then there is 
a triangle equivalence 
$$\underline{\mcX^{\perp_1}}\cong \mcX^{\perp_1}_{\underline{\mathcal{F}}}/\add(\mcX)$$
between the stable category of the reduction of $\mathcal{F}$ at $\mathcal{X}$, and the Iyama–Yoshino reduction of
the stable category $\mathcal{F}$ at $\mcX$.
\end{teo}

Later in \cite{huerta2024reduction}, the reduction technique in \cite{FMP23} was
extended.
This approach has an advantage in comparison with \cite{FMP23} since it can be applied for non Frobenius extriangulated categories. 
However, the reduction category in the latter does not necessarily admit a triangulated structure in 
contrast with Theorem~\ref{teo 1}. From this,
a natural question arises: how far the Iyama-Yoshino reduction can be carried on non Frobenius
extriangulated categories? The main of this work is to study an interpretation of Iyama-Yoshino reduction and
provide an analogue of Theorem~\ref{teo 1} for the non Frobenius case.

\subsection*{Organization of the paper} 

We begin Section~\ref{sec:preliminaries} by establishing notation, definitions and outcomes in extriangulated categories needed throughout this work. In Section~\ref{sec: quotients} we address quotients categories 
for which it is known that an extriangulated category is 
admitted. We present our main result in 
Theorem~\ref{teo: equiv extriang} where we give an equivalence of extriangulated categories between certain quotient categories determined by the class of projective-injective 
objects (see Remark~\ref{rmk: equiv extriang}). Section~\ref{sec: examples} is devoted to present examples of the equivalence given in Section~\ref{sec: quotients}. Finally, Section~\ref{sec: compatibility with triangulated} shows an example where these quotient
categories may not admit a triangulated structure in 
contrast to triangulated or Frobenius extriangulated cases (see Example~\ref{ex: red no frobenius}).

%%%%%%%%%%%%%%%%%%%%%%%%%%%%%%%%%%%%%
%%%%%%%%%%%%%%%%%%%%%%%%%%%%%%%%%%%%%

\subsection*{Conventions}
Throughout this article, $\mcC$ denotes an additive category. 
The main examples will be categories of finitely generated left $\Lambda$-modules over an Artin algebra $\Lambda$ which we denote by $\fgMod(\Lambda)$. We denote by 
$\Obj (\mcC)$ the class of objects in $\mcC$ and by
$\Mor (\mcC)$ the class of all the morphisms in $\mcC$. On the other hand, we write $\mcS \subseteq \mcC$ to say that $\mcS$ is a full subcategory of $\mcC$. All the class of objects in $\mcC$ are assumed to be full subcategories. Given $X, Y \in \mcC$, we denote by $\mcC(X,Y)$ or by $\Hom_\mcC (X, Y)$ the group of morphisms from $X$ to $Y$. In case $X$ and $Y$ are isomorphic, we write $X \simeq Y$. The notation $F \cong G$, on the other hand, is reserved to denote the existence of a natural isomorphism between functors $F$ and $G$.
Given a class $\mcA \subseteq \mcC$, we denote by ${\rm free}(\mathcal{A})$ the class of all the finite coproducts of objects in $\mcA,$ and $\smd(\mcA)$  the class of all the direct summands of objects in $\mcA.$ We set $\add(\mcA):=\smd(\free(\mcA)).$ Finally, we write $[1,n]$ to denote the set of the first $n$ natural numbers, for any $n\geq 1$.

%%%%%%%%%%%%%%%%%%%%%%%%%%%%%%%%%%%%%
%%%%%%%%%%%%%%%%%%%%%%%%%%%%%%%%%%%%%
%%%%%%%%%%%%%%%%%%%%%%%%%%%%%%%%%%%%%
%%%%%%%%%%%%%%%%%%%%%%%%%%%%%%%%%%%%%

\section{\textbf{Preliminaries}}\label{sec:preliminaries}

\subsection*{Extriangulated categories and terminology}

Now we recall some definitions and results related to extriangulated categories. 
For a detailed treatise on this matter, we recommend  the reader to see in
\cite{NP19, LNheartsoftwin, MDZtheoryAB}.

Let $\mathbb{E}:\mcC^{op}\times \mcC\to \mathrm{Ab}$ be an additive bifunctor. An $\mathbb{E}$-extension \cite[Definition 2.1 and Remark 2.2]{NP19} is a triplet $(A, \delta, C),$ where $A, C\in \mcC$ and $\delta\in \mathbb{E}(C, A).$ 
For any $a\in \mcC(A, A')$
and $c\in \mcC(C', C)$, we have $\mathbb{E}$-extensions $a\cdot \delta:=\mathbb{E}(C, a)(\delta)\in \mathbb{E}(C, A')$ and
$\delta \cdot c:=\mathbb{E}(c^{op}, A)(\delta)\in \mathbb{E}(C', A)$. In this terminology, we have
$(a\cdot \delta)\cdot c=a\cdot(\delta\cdot c)$ in $\mathbb{E}(C', A')$. Let $(A, \delta, C)$ and  $(A', \delta', C')$ be $\mathbb{E}$-extensions. A morphism $(a, c): 
(A, \delta, C)\to (A', \delta', C')$ of $\mathbb{E}$-extensions \cite[Definition 2.3]{NP19} is a pair of morphisms $a\in \mcC(A, A')$ and
$c\in \mcC(C, C')$ in $\mcC$, satisfying the equality $a\cdot \delta=\delta'\cdot c.$
We simply denote it as $(a, c): \delta\to \delta'$.
We obtain the category $\mathbb{E}$-$\Ext(\mcC)$ of $\mathbb{E}$-extensions, with composition and identities
naturally induced by those in $\mcC$. For any $A, C\in \mcC$, the zero element $0\in \mathbb{E}(C, A)$ is called the split $\mathbb{E}$-extension \cite[Definition 2.5]{NP19}.
\

Let $\delta=(A, \delta, C)$ and $\delta'=(A', \delta', C')$ be any $\mathbb{E}$-extensions, and let 
$C\mathop{\to}\limits^{\iota_{C}} C\oplus C\mathop{\leftarrow}\limits^{\iota_{C'}} C'$ and 
$A\mathop{\leftarrow}\limits^{p_{A}} A\oplus A'\mathop{\to}\limits^{p_{A'}}A'$ be coproduct and product in $\mcC$,
respectively. By the biadditivity of $\mathbb{E}$, we have a natural isomorphism
$$\mathbb{E}(C\oplus C', A\oplus A')\cong \mathbb{E}(C, A)\oplus \mathbb{E}(C, A')\oplus 
\mathbb{E}(C', A)\oplus \mathbb{E}(C', A').$$
Following \cite[Definition 2.6]{NP19}, let $\delta\oplus \delta'\in \mathbb{E}(C\oplus C', A\oplus A')$ be the element corresponding to $(\delta, 0, 0, 
\delta')$ through this isomorphism.
If $A=A'$ and $C=C'$, then the sum $\delta+\delta'\in \mathbb{E}(C, A)$ of $\delta, \delta'\in 
\mathbb{E}(C, A)$ is obtained by 
$$\delta+\delta'=\nabla_{A}\cdot (\delta\oplus \delta')\cdot\Delta_{C}$$
where $\Delta_{C}=\tiny{\left( 
\begin{array}{c}
1_C\\1_C
\end{array}
\right)} : C\to C\oplus C$ and $\nabla_{A}=(1_A\, 1_A): A\oplus A\to A$.
\

Two sequences of morphisms $A\mathop{\to}\limits^{x} B
\mathop{\to}\limits^{y} C$ and $A\mathop{\to}\limits^{x'} B'\mathop{\to}\limits^{y'} C$ in $\mcC$ are said to be
equivalent \cite[Definition 2.7]{NP19} if there exists an isomorphism $b\in \mcC(B, B')$ which makes the following diagram commutative
\[
\xymatrix@R=4mm{
& B\ar[dr]^{y}\ar[dd]^{b}_{\wr} &\\
A\ar[ur]^{x}\ar[dr]_{x'} & & C.\\
& B'\ar[ur]_{y'} & 
}
\]
We denote the equivalence class of $A\mathop{\to}\limits^{x} B\mathop{\to}\limits^{y} C$ by 
$[A\mathop{\to}\limits^{x} B\mathop{\to}\limits^{y} C]$. Moreover, 
for any two classes $[A\mathop{\to}\limits^{x} B\mathop{\to}\limits^{y} C]$ and
$[A'\mathop{\to}\limits^{x'} B'\mathop{\to}\limits^{y'} C']$, we set \cite[Definition 2.8]{NP19}
\begin{center}
$[A\mathop{\to}\limits^{x} B\mathop{\to}\limits^{y} C]\oplus [A'\mathop{\to}\limits^{x'} B'\mathop{\to}\limits^{y'} C']:=[A\oplus A'\mathop{\to}\limits^{x\oplus x'} B\oplus B'\mathop{\to}\limits^{y\oplus
y'} C\oplus C'].$
\end{center}

\begin{definition}\cite[Definition 2.9]{NP19}\label{def 2.9}
Let $\mathfrak{s}$ be a correspondence which associates to each $\mathbb{E}$-extension $\delta\in \mathbb{E}(C, A)$ an equivalence class $\mathfrak{s}(\delta)=
[A\mathop{\to}\limits^{x} B\mathop{\to}\limits^{y} C]$. This $\mathfrak{s}$ is called a realization of $\mathbb{E}$ if it satisfies the following condition:

$(*)$ Let $\delta\in \mathbb{E}(C, A)$ and $\delta'\in \mathbb{E}(C', A')$ be any pair of $\mathbb{E}$-extensions, with $\mathfrak{s}(\delta)=[A\mathop{\to}\limits^{x} B\mathop{\to}\limits^{y} C]$ and
$\mathfrak{s}(\delta')=[A'\mathop{\to}\limits^{x'} B'\mathop{\to}\limits^{y'} C']$. Then, for any morphism 
$(a, c)\in \mathbb{E}\mbox{-}\Ext(\mcC)(\delta, \delta')$, there exists $b\in \mcC(B, B')$ which makes the 
following diagram commutative
\begin{equation}\label{eq: diag1}
\xymatrix{
A\ar[r]^{x}\ar[d]_{a} & B\ar[r]^{y}\ar[d]^{b} & C\ar[d]^{c}\\
A'\ar[r]_{x'} & B'\ar[r]_{y'} & C'.
}
\end{equation}
It is said that the sequence $A\mathop{\to}\limits^{x} B\mathop{\to}\limits^{y} C$ realizes $\delta$ if $\mathfrak{s}(\delta)=[A\mathop{\to}\limits^{x} B\mathop{\to}\limits^{y} C]$. We point out that this condition does not depend on the choices of the representatives of the equivalence classes. 
In the above situation, we say that \eqref{eq: diag1} (or the triplet $(a, b, c)$) realizes $(a, c)$.
\end{definition}

A realization $\mathfrak{s}$ of $\mathbb{E}$ is additive \cite[Definitions 2.8 and 2.10]{NP19} if it satisfies the following two conditions:
\begin{enumerate}
\item $\mathfrak{s}(0)=\left[A\mathop{\to}\limits^{\tiny{\left(
\begin{array}{c}
1\\ 0
\end{array}
\right)}} A\oplus C\mathop{\to}\limits^{(0\, 1)} C
\right]$ for any $A, C\in \mcC;$
\item $\mathfrak{s}(\delta\oplus \delta')=\mathfrak{s}(\delta)\oplus \mathfrak{s}(\delta')$, for
any $\mathbb{E}$-extensions $\delta$ and $\delta'$.
\end{enumerate}

\begin{definition}\cite[Definition 2.12]{NP19}
The pair $(\mathbb{E}, \mathfrak{s})$ is an external triangulation of $\mcC$ if it satisfies the following
conditions.
\begin{enumerate}
\item[(ET1)] $\mathbb{E}: \mcC^{op}\times \mcC\to \mathrm{Ab}$ is an additive bifunctor.
\item[(ET2)] $\mathfrak{s}$ is an additive realization of $\mathbb{E}$.
\item[(ET3)] Let $\delta\in \mathbb{E}(C, A)$ and $\delta'\in \mathbb{E}(C', A')$ be any pair of $\mathbb{E}$-extensions, realized as $\mathfrak{s}(\delta)=[A\mathop{\to}\limits^{x} B\mathop{\to}\limits^{y} C]$ and
$\mathfrak{s}(\delta')=[A'\mathop{\to}\limits^{x'} B'\mathop{\to}\limits^{y'} C']$. For any commutative square
in $\mcC$
\[
\xymatrix{
A\ar[r]^{x}\ar[d]_{a} & B\ar[r]^{y}\ar[d]^{b} & C\\
A'\ar[r]_{x'} & B'\ar[r]_{y'} & C',
}
\]
there exists a morphism $(a, c): \delta\to \delta'$ which is realized by $(a, b, c)$.
\item[(ET3)$^{op}$] Let $\delta\in \mathbb{E}(C, A)$ and $\delta'\in \mathbb{E}(C', A')$ be any pair of 
$\mathbb{E}$-extensions, realized by $A\mathop{\to}\limits^{x} B\mathop{\to}\limits^{y} C$ and 
$A'\mathop{\to}\limits^{x'} B'\mathop{\to}\limits^{y'} C'$, respectively. For any commutative square in $\mcC$
\[
\xymatrix{
A\ar[r]^{x} & B\ar[r]^{y}\ar[d]_{b} & C\ar[d]^{c}\\
A'\ar[r]_{x'} & B'\ar[r]_{y'} & C'
}
\]
there exists a morphism $(a, c): \delta\to \delta'$ which is realized by $(a, b, c)$.

\item[(ET4)] Let $(A, \delta, D)$ and $(B, \delta', F)$ be $\mathbb{E}$-extensions realized, respectively, by
$A\mathop{\to}\limits^{f} B\mathop{\to}\limits^{f'} D$ and $B\mathop{\to}\limits^{g} C\mathop{\to}\limits^{g'}
F$. Then there exist an object $E\in \mcC$, a commutative diagram
\[
\xymatrix{
A\ar@{=}[d]\ar[r]^{f} & B\ar[r]^{f'}\ar[d]_{g} & D\ar[d]^{d}\\
A\ar[r]_{h} & C\ar[r]_{h'}\ar[d]_{g'} & E\ar[d]^{e}\\
& F\ar@{=}[r] & F
}
\]
in $\mcC$, and an $\mathbb{E}$-extension $\delta''\in \mathbb{E}(E, A)$ realized by $A\mathop{\to}\limits^{h} C
\mathop{\to}\limits^{h'} E$, which satisfy $\mathfrak{s}(f'\cdot \delta ')=[D\mathop{\to}\limits^{d} E\mathop{\to}\limits^{e} F]$, $\delta''\cdot d=\delta$ and $f\cdot \delta''=\delta'\cdot e$.

\item[(ET4)$^{op}$] Let $(D, \delta, B)$ and $(F, \delta', C)$ be $\mathbb{E}$-extensions realized, respectively, by $D\mathop{\to}\limits^{f'} A\mathop{\to}\limits^{f} B$ and
$F\mathop{\to}\limits^{g'} B\mathop{\to}\limits^{g} C$, respectively. Then there
exist an object $E\in \mcC$, a commutative diagram
\[
\xymatrix{
D\ar@{=}[d]\ar[r]^{d} & E\ar[r]^{e}\ar[d]_{h'} & F\ar[d]^{g'}\\
D\ar[r]_{f'} & A\ar[r]_{f}\ar[d]_{h} & B\ar[d]^{g}\\
& C\ar@{=}[r] & C
}
\]
in $\mcC$ and an $\mathbb{E}$-extension $\delta''\in \mathbb{E}(C, E)$ realized by
$E\mathop{\to}\limits^{h'} A\mathop{\to}\limits^{h} C$ which satisfy $\mathfrak{s}(\delta\cdot  g')=[D\mathop{\to}\limits^{d} E\mathop{\to}\limits^{e} F]$, $\delta'=e\cdot \delta''$ and 
$d\cdot \delta=\delta''\cdot g$.
\end{enumerate}
If the above conditions hold true, we call $\mathfrak{s}$ an $\mathbb{E}$-triangulation of $\mcC$, and call the triplet
$(\mcC, \mathbb{E}, \mathfrak{s})$ an externally triangulated category, or for short, extriangulated 
category. Sometimes, for the sake of
simplicity, we only write $\mcC$ instead of $(\mcC, \mathbb{E}, \mathfrak{s}).$ 
\end{definition}

For a triplet $(\mcC, \mathbb{E}, \mathfrak{s})$ satisfying (ET1) and (ET2), we recall that \cite[Definition 2.15]{NP19}:
\begin{enumerate}
\item A sequence $A\mathop{\to}\limits^{x} B\mathop{\to}\limits^{y} C$ is called a $\mbE$-conflation if it
realizes some $\mathbb{E}$-extension $\delta\in \mathbb{E}(C, A)$.
\item A morphism $f\in\mcC(A, B)$ is called an $\mbE$-inflation if it admits some $\mbE$-conflation 
$A\mathop{\to}\limits^{f} B\to C$.
\item A morphism $f\in\mcC(A,B)$ is called a $\mbE$-deflation if it admits some $\mbE$-conflation $K\to A\mathop{\to}\limits^{f} B$.
\end{enumerate}

Recall from \cite[Definition 2.17]{NP19}, that  a subcategory $\mcD\subseteq \mcC$ with $\mathcal{C}$ extriangulated category is \emph{closed under extensions} if, for any conflation $A\to B\to C$ with $A, C\in \mcD$, we have $B\in \mcD$.
\\

Let $(\mcC, \mathbb{E}, \mathfrak{s})$ be a triplet satisfying (ET1) and (ET2). Then, by following \cite[Definition   2.19]{NP19}, we have that: 
\begin{enumerate}
\item If $A\mathop{\to}\limits^{x} B\mathop{\to}\limits^{y} C$ realizes $\delta\in \mathbb{E}(C, A)$, we call the pair $(A\mathop{\to}\limits^{x} B\mathop{\to}\limits^{y} C, \delta)$ an $\mathbb{E}$-triangle,
and we write it as
$\xymatrix{A\ar[r]^{x} & B\ar[r]^{y} & C\ar@{-->}[r]^{\delta} & }.$ Let us consider another $\mathbb{E}$-triangle $(A'\mathop{\to}\limits^{x'} B'\mathop{\to}\limits^{y'} C', \delta').$ Then, the fact that $\mathfrak{s}$ is an additive realization of $\mbE$ give us the $\mbE$-triangle
\[\xymatrix{A\oplus A'\ar[r]^{x\oplus x'} & B\oplus B'\ar[r]^{y\oplus y'} & C\ar@{-->}[r]^{\delta\oplus\delta'} & }.\]

\item Let $A\mathop{\to}\limits^{x} B\mathop{\to}\limits^{y} C\mathop{\dashrightarrow}\limits^{\delta}$ and
$A'\mathop{\to}\limits^{x'} B'\mathop{\to}\limits^{y'} C'\mathop{\dashrightarrow}\limits^{\delta'}$ be  $\mathbb{E}$-triangles. If a triplet $(a, b, c)$ realizes $(a, c): \delta\to \delta'$ as in Definition \ref{def 2.9}, then we write
it as
\[
\xymatrix{
A\ar[r]^{x}\ar[d]_{a} & B\ar[r]^{y}\ar[d]^{b} & C\ar@{-->}[r]^{\delta}\ar[d]^{c} &\\
A'\ar[r]_{x'} & B'\ar[r]_{y'} & C'\ar@{-->}[r]_{\delta'} &
}
\]
and we call $(a, b, c)$ a morphism of $\mathbb{E}$-triangles.
\end{enumerate}

Let $\mathbb{E}:\mcC^{op}\times\mcC\to Ab$ be an additive bifunctor.
By Yoneda's lemma and \cite[Definition\ 3.1]{NP19}, any $\mathbb{E}$-extension
$\delta\in \mathbb{E}(C, A)$ induces natural
transformations $\delta_{\#}: \mcC(-,C)\rightarrow \mathbb{E}(-,A)$ and 
$\delta^{\#}: \mcC(A, -)\rightarrow \mathbb{E}(C,-)$. For any $X\in \mcC$, these $(\delta_{\#})_{X}$ and $\delta^{\#}_{X}$ are given as follows
\begin{enumerate}
\item $(\delta_{\#})_{X}: \mcC(X, C)\to 
\mathbb{E}(X, A); f\mapsto \delta\cdot f;$
\item $\delta^{\#}_{X}: \mcC(A, X)\to \mathbb{E}(C, X); g\mapsto g\cdot \delta$.
\end{enumerate}
We abbreviately denote $(\delta_{\#})_{X}(f)$
and $\delta^{\#}_{X}(g)$ by $\delta_{\#}f$ and
$\delta^{\#}g$, respectively.

\begin{corollary}\cite[Corollary 3.12]{NP19}\label{suc exact ext1}
Let $\mcC$ be an extriangulated category. For any 
$\mathbb{E}$-triangle $A\mathop{\to}\limits^{x} B\mathop{\to}\limits^{y} C \mathop{\dashrightarrow}\limits^{\delta}$, we have the 
following exact sequences of additive functors
$$\mcC(C,-)\mathop{\longrightarrow}\limits^{\mcC(y,-)} \mcC(B,-)
\mathop{\longrightarrow}\limits^{\mcC(x,-)} \mcC(A,-)\mathop{\longrightarrow}\limits^{\delta^{\#}} \mathbb{E}(C,-)
\mathop{\longrightarrow}\limits^{\mathbb{E}(y,-)} \mathbb{E}(B,-)
\mathop{\longrightarrow}\limits^{\mathbb{E}(x,-)} \mathbb{E}(A,-),$$
$$\mcC(-,A)\mathop{\longrightarrow}\limits^{\mcC(-,x)} \mcC(-,B)
\mathop{\longrightarrow}\limits^{\mcC(-,y)} \mcC(-,C)\mathop{\longrightarrow}\limits^{\delta_{\#}} \mathbb{E}(-,A)
\mathop{\longrightarrow}\limits^{\mathbb{E}(-,x)} \mathbb{E}(-,B)
\mathop{\longrightarrow}\limits^{\mathbb{E}(-,y)} \mathbb{E}(-,C).$$
\end{corollary}

\subsection*{Higher extensions} Let $\mcC$ be an extriangulated category. Following \cite{NP19}, we recall that an object $P\in \mcC$ is
$\mbE$-projective if for any $\mathbb{E}$-triangle $A\mathop{\to}\limits^{x} B
\mathop{\to}\limits^{y} C\mathop{\dashrightarrow}\limits^{\delta}$ the map
$$\mcC(P,y): \mcC(P,B)\longrightarrow \mcC(P,C)$$
is surjective. We denote by $\mathcal{P}(\mcC)$ the class of $\mbE$-projective objects in $\mcC$. Dually, the class of $\mbE$-injective objects in $\mcC$ is denoted by $\mcI(\mcC)$. We say that $\mcC$ has \emph{enough $\mbE$-projectives} if for any object $C\in \mcC$, there exists an 
$\mathbb{E}$-triangle $A\to P\to C\dashrightarrow$ with $P\in \mcP(\mcC)$. Dually, we can
define that $\mcC$ has \emph{enough $\mbE$-injectives}.
\

\begin{lemma}\cite[Proposition 3.24]{NP19}\label{lem: E-projective=projective}
Let $\mcC$ be an extriangulated category. An object $P\in \mcC$ is $\mbE$-projective in $\mcC$
if and only if $\mbE(P,C)=0$ for all $C\in \mcC$.
\end{lemma}

\noindent Given $\mcX, \mcY\subseteq \mcC$ classes of objects in an extriangulated category $\mcC$, we recall 
the following from \cite[Definition 4.2]{NP19}:
\begin{enumerate}
\item[$\bullet$] $C\in \mcC$ belongs to $\mathrm{Cone}(\mcX, \mcY)$ if $C$ admits a 
conflation $X\to Y\to C$ with $X\in \mcX, Y\in \mcY$.

\item[$\bullet$] $C\in \mcC$ belongs to $\mathrm{CoCone}(\mcX, \mcY)$ if $C$ admits a 
conflation $C\to X\to Y$ with $X\in \mcX, Y\in \mcY$.

\item[$\bullet$] $\mcX$ is \emph{closed under cones} if
$\mathrm{Cone}(\mcX, \mcX)\subseteq \mcX$.
Dually, $\mcX$ is \emph{closed
under cocones} if $\mathrm{CoCone}(\mcX, \mcX)\subseteq \mcX$.
\end{enumerate}

We set $\Omega\mcX:=\mathrm{CoCone}(\mcP(\mcC), \mcX)$, that is, $\Omega \mcX$ is
the subclass of $\mcC$ consisting of the objects $\Omega X$ admitting an $\mathbb{E}$-triangle
$\Omega X\to P\to X\dashrightarrow$
with $P\in\mcP(\mcC)$ and $X\in \mcX$. We call $\Omega\mcX$ the syzygy class of $\mcX$ in $\mcC$. We set $\Omega^{0}\mcX:=\mcX$, and define $\Omega^{k}\mcX$ for $k>0$
inductively by
$\Omega^{k}\mcX:=\Omega(\Omega^{k-1}\mcX)$ which is the $k$-th syzygy class of $\mcX$. Dually, 
the cosyzygy class of $\mcX$ is $\Sigma\mcX:=\mathrm{Cone}(\mcX, \mcI(\mcC))$ and $\Sigma^{k}\mcX$ is the $k$-th cosyzygy class of $\mcX,$ for $k\geq 0$  
(see \cite[Definition 4.2 and Proposition 4.3]{LNheartsoftwin},
for more details).\\

\noindent Let $\mcC$ be an extriangulated category with enough $\mbE$-projectives and $\mbE$-injectives.
In \cite{LNheartsoftwin}, it is shown that $\mathbb{E}(X, \Sigma^{k}Y)\cong \mathbb{E}(\Omega^{k}X, Y)$ 
 for $k\geq 0.$ Thus, the higher extension groups are defined as 
 \begin{equation}\label{def: higher ext}
     \mathbb{E}^{k+1}(X, Y):=\mathbb{E}(X, \Sigma^{k}Y)\cong \mathbb{E}(\Omega^{k}X, Y),
 \end{equation}
  for $k\geq 0$. 
 Moreover, the following result is also proven.

\begin{lemma}\cite[Proposition 5.2]{LNheartsoftwin} \label{longExSeq}
Let $\mcC$ be an extriangulated category with enough $\mbE$-projectives and $\mbE$-injectives and $A\to B\to C\dashrightarrow$ be an $\mathbb{E}$-triangle in $\mcC$. Then, for any object $X\in \mcC$ and $k\geq 1$, we have the following
exact sequences 
\[(1)\;
\mathbb{E}^{k}(X, A)\to \mathbb{E}^{k}(X, B)\to \mathbb{E}^{k}(X, C)\to \mathbb{E}^{k+1}(X, A)\to \mathbb{E}^{k+1}(X, B)\to \cdots,
\]
\[(2)\;
\mathbb{E}^{k}(C, X)\to \mathbb{E}^{k}(B, X)\to \mathbb{E}^{k}(A, X)\to \mathbb{E}^{k+1}(C, X)\to \mathbb{E}^{k+1}(B, X)\to \cdots
\]
of abelian groups.
\end{lemma}

Let $\mcC$ be an extriangulated category with 
enough $\mbE$-projectives and $\mbE$-injectives. We fix the following notation for $\mcX, \mcY\subseteq \mcC$ and $k \geq 1$.
\begin{itemize}
\item $\mathbb{E}^{k}(\mathcal{X, Y})=0$ if 
$\mathbb{E}^{k}(X, Y)=0$ for every $X\in \mcX$ and $Y\in \mcY$. When $\mcX=\{M\}$ or $\mcY=\{N\}$, 
we shall write $\mathbb{E}^{k}(M, \mcY)=0$ and
$\mathbb{E}^{k}(\mcX, N)=0$, respectively.

\item $\mathbb{E}^{\leq k}(\mcX, \mcY)=0$ if
$\mathbb{E}^{j}(\mcX, \mcY)=0$ for every $1\leq j\leq k$.

\item $\mathbb{E}^{\geq k}(\mcX, \mcY)=0$ if 
$\mathbb{E}^j(\mcX, \mcY)=0$ for every $j\geq k$.
\end{itemize}

 Recall that the \emph{right $k$-th orthogonal complement} and the \emph{right orthogonal complement of $\mcX$} are defined, respectively, by 
\begin{align*}
\mcX^{\perp_k} := \{ N \in \mcC \mbox{ : } \mathbb{E}^k(\mcX,N) = 0 \}\;\text{ and }\; \mcX^\perp := \bigcap_{k \geq 1} \mcX^{\perp_k}=\{N\in \mcC : \mathbb{E}^{\geq 1}(\mcX, N)=0\}.
\end{align*}
 Dually, we have the \emph{left $k$-th} and the \emph{left orthogonal complements ${}^{\perp_k}\mcX$ and ${}^{\perp}\mcX$ of $\mcX$}, respectively. 

\subsection*{Quotient categories}
Let $\mcC$ be an additive category and 
$I\subseteq \Mor(\mcC)$ be an ideal of $\mcC$. The \emph{quotient of $\mcC$ by $I$}, denoted by
$\mcC/I$,
 is the category 
whose objects, morphisms and composition are defined by:

\begin{itemize}
\item[$\bullet$] Objects: $\Obj(\mcC/I) := \Obj(\mcC)$.  

\item[$\bullet$] Morphisms: for each pair $(A,B) \in \Obj(\mcC)\times \Obj(\mcC)$, 
	\[
	\Hom_{\mcC/I}(A,B) := \frac{\Hom_{\mcC}(A,B)}{I(A,B)}
	\]
	
\item[$\bullet$] Composition of morphisms in $\mcC/I$:
	$$\begin{array}{ccc}
\Hom_{\mcC/I}(B, C)\times \Hom_{\mcC/I}(A, B) &
\longrightarrow & \Hom_{\mcC/I}(A, C)\\
{}\\
(g+I(B,C), f+I(A, B)) & \mapsto & gf+I(A, C)
\end{array}
$$
\end{itemize}
We
denote
by $\pi_{I}: \mcC\longrightarrow \mcC/I$
the canonical projection given by 
$\pi_{I}(M):=M$, for all $M\in \Obj(\mcC)$ and 
$\pi_{I}(f):=f+I(A, B)$, for any morphism $f:A\to B\in \Mor(\mcC)$. Notice that $\pi_{I}:\mcC\to \mcC/I$ is a full and
essentially surjective functor. Given a subcategory $\mcY\subseteq \mcC$, we denote by $[\mcY]$
the class of morphisms in $\mcC$ such that factor through
an object in $\mcY$. Notice that, in case $\free(\mcY)=\mcY$, $[\mcY]$ is an ideal of $\mcC$.

\section{\textbf{Quotient categories with extriangulated structure}}\label{sec: quotients}

We begin this section with the following well-known result.

\begin{proposition}\cite[Proposition 3.30]{NP19}\label{estruct del cociente}
Let $(\mcC, \mbE, \mathfrak{s})$ be an extriangulated category and $J=\free (J)$. If $J\subseteq \mcP(\mcC)\cap \mcI(\mcC)$ then the 
quotient category $\mcC_J:=\mcC/[J]$ has an extriangulated structure $(\mcC_J, \mbE_{\mcC_J}, \mathfrak{s}_{\mcC_J})$ induced from the one of 
$\mcC$.
\end{proposition}

Moreover, the extriangulated structure described in \cite[Proposition 3.30]{NP19} is given as follows.\\

Let $\mcC_J:=\mcC/[J]$ and $\pi_J:\mcC\to \mcC_J$ be the corresponding canonical projection. Then,

\begin{enumerate}
    \item[$(I)$]  For any $M, N\in \Obj(\mcC)$ and any $f, g\in \Mor(\mcC)$, 
\begin{equation}\label{eq: estruct del cociente}
\mbE_{\mcC_J}(M, N):=\mbE(M, N) \quad \mbox{ and } \quad \mbE_{\mcC_J}((\pi_J(f))^{op}, \pi_J(g) ):=\mbE(f^{op}, g).
\end{equation}

\item[$(II)$] For any $\delta\in \mbE_{\mcC_J}(C, A)$, 
$$\mathfrak{s}_{\mcC_J}(\delta):=
[A\mathop{\longrightarrow}\limits^{\pi_J(x)} B \mathop{\longrightarrow}\limits^{\pi_J(y)} C],$$
where $\mathfrak{s}(\delta)=[A\mathop{\longrightarrow}\limits^{x} B\mathop{\longrightarrow}\limits^{y} C]$.
\end{enumerate}

\begin{notation}
For any $\delta\in \mbE_{\mcC_J}(C, A)=\mbE(C, A)$, we use the notation $\pi_J (\delta)$ to distinguish $\delta$ as element of $\mbE_{\mcC_J}(C, A)$ and we only write $\delta$ when we consider $\delta$
as element in $\mbE(C, A)$. 
Given 
$\pi_J(\delta)\in \mbE_{\mcC_J}(C, A)$, we write   
$$A\mathop{\longrightarrow}\limits^{\pi_J(x)} B\mathop{\longrightarrow}\limits^{\pi_J(y)} C\mathop{\dashrightarrow}\limits^{\pi_J(\delta)}$$ to denote its corresponding 
$\mbE_{\mcC_J}$-triangle in $\mcC_J$.
\end{notation}

Below we describe the existing relation between left and right operations $\cdot$ of 
$\mbE$ in $\mcC$ and whose of $\mbE_{\mcC_J}$ in $\mcC_J$.

\begin{lemma}\label{lem: pi y cdot}
Let $\mcC$ be an extriangulated category and   
$\free (J)=J\subseteq \mcC$ satisfying $J\subseteq \mcP(\mcC)\cap \mcI(\mcC)$. Let $\mcC_J:=\mcC/[J]$ and 
$\pi_J:\mcC\to \mcC_J$ be the corresponding canonical
projection. Then, $$\pi_J(f)\cdot \pi_J(\delta)\cdot \pi_J(g)=\pi_J (f\cdot \delta\cdot g),$$
for
any $\mbE$-triangle $A\to B\to C\mathop{\dashrightarrow}\limits^{\delta}$ and for
any morphisms $f:A\to A'$ and $g:C'\to C$ in $\mcC$. 
\end{lemma}

\begin{proof}
Let $f:A\to A'$ and $g:C'\to C$ be two morphisms in $\mcC$.
Then, by \eqref{eq: estruct del cociente}, we have
\[
\begin{array}{rl}
\pi_J(f)\cdot \pi_J(\delta)\cdot \pi_J(g)= & \mbE_{\mcC_J}((\pi_J(g))^{op}, \pi_J(f))(\pi_J(\delta))\\
= & \mbE(g^{op}, f)(\delta)\\
= & f\cdot \delta\cdot g.
\end{array}
\]
It means that $\pi_J(f)\cdot \pi_J(\delta)\cdot \pi_J(g)$
seen as element of $\mathbb{E}_{\mcC_J}(C', A')$ corresponds to 
$f\cdot \delta\cdot g$ as element of $\mbE(C', A')$.
Hence, the equality 
$\pi_J(f\cdot \delta\cdot g)=\pi_J(f)\cdot \pi_J(\delta)\cdot \pi_J(g)$ holds.
\end{proof}

The following proposition shows us that the property of having enough $\mbE$-projectives and $\mbE$-injectives is preserved and  therefore higher extension groups as well.

\begin{proposition}\label{prop: quot C suf proj}
Let $\mcC$ be an extriangulated category with enough 
$\mbE$-projectives and $\mbE$-injectives, and 
$\free (J)=J\subseteq \mcC$ satisfying $J\subseteq \mcP(\mcC)\cap \mcI(\mcC)$. Then, the following statements 
hold true, for $\mcC_J:=\mcC/[J]$ and $\pi_J:\mcC\to \mcC_J$ being the corresponding canonical projection.
\begin{enumerate}[(a)]
\item $\mcC_J$ has enough 
$\mbE_J$-projectives and
$\mbE_J$-injectives. Moreover,
$\pi_J(\mcP(\mcC))= \mcP(\mcC_J)$ and $\pi_J(\mcI(\mcC))= \mcI(\mcC_J)$.

\item $\mbE^{i}_{\mcC_J}(M,N)\cong \mbE^{i}(M, N)$,
for any $M, N\in \mcC$ and for any $i\geq 1$.
\end{enumerate}
\end{proposition}

\begin{proof}
(a) From \eqref{eq: estruct del cociente}, Lemma~\ref{lem: E-projective=projective} and its dual it is clear
that the equalities
\[
\pi_J(\mcP(\mcC))=\mcP(\mcC_J) \quad \mbox{ and }\quad 
\pi_J(\mcI(\mcC))=\mcI(\mcC_J)
\]
hold true. Now, let $M\in \Obj(\mcC_J)=\Obj(\mcC)$. Since
$\mcC$ has enough $\mbE$-projectives, there exists an $\mbE$-triangle
$\Omega M\mathop{\longrightarrow}\limits^{f} P_0\mathop{\longrightarrow}\limits^{g} M\mathop{\dashrightarrow}\limits^{\delta}$ in $\mcC$ with $P_0\in \mcP(\mcC)$. So,
\begin{equation}\label{E-tr en C cociente}
\Omega M\mathop{\longrightarrow}\limits^{\pi_J(f)} P_0\mathop{\longrightarrow}\limits^{\pi_J (g)}M\mathop{\dashrightarrow}\limits^{\pi_J(\delta)}
\end{equation}
is 
an $\mbE_{\mcC_J}$-triangle in $\mcC_J$ with $P_0=\pi_J(P_0)\in \pi_J(\mcP(\mcC))=\mcP(\mcC_J)$. This shows that 
$\mcC_J$ has enough $\mbE_{\mcC_J}$-projectives. Dually,
one can prove that 
$\pi_J(\mcI(\mcC))=\mcI(\mcC_J)$ and $\mcC_J$ has
enough $\mbE_{\mcC_J}$-injectives.\\

(b) The case $i=1$ is clear due 
to~\eqref{eq: estruct del cociente}. So, we can assume that $i\geq 2$.

Let $M, N\in 
\Obj(\mcC_J)=\Obj(\mcC)$. 
By using that $\mcC$ has enough
$\mbE$-projectives we get an $\mbE_{\mcC_J}$-triangle in $\mcC_J$ as in~\eqref{E-tr en C cociente} with $P_0\in \mcP(\mcC_J)$. Thus, 
it follows that 
\[
\mbE^{2}_{\mcC_J}(M,N)\cong \mbE_{\mcC_J}(\Omega M, N)
\]
Inductively, for every
$i\geq 1$, we have 
\begin{equation}\label{eq1}
\mbE_{\mcC_J}^{i+1}(M, N)\cong \mbE_{\mcC_J} (\Omega^{i} M, N)
\end{equation}
Hence, by (a), we get 
\[
\mbE^{i+1}_{\mcC_J}(M, N)\cong \mbE_{\mcC_J}(\Omega^{i} M, N)=
 \mbE(\Omega^{i} M, N)\cong \mbE^{i+1}(M, N).
\]
\end{proof}

In the sequel, we write 
$\mcY^{\perp_{i, J}}$ and
${}^{\perp_{i, J}}\mcY$ to 
denote, respectively, the right and the left $i$th-orthogonal complement  
for any subcategory $\mcY\subseteq \mcC_J$ and any
$i\geq 1$. We also write $\pi_J(\mcY)\subseteq \mcC_J$ to denote the essential image of $\mcY\subseteq \mcC$ under $\pi_J$, 
that is, 
\[
\pi_J(\mcY):=\{N\in \mcC_J : N\simeq \pi_J(Y) 
\mbox{ in } \mcC_J \mbox{ for some } Y\in \mcY\}.
\]

\begin{lemma}\label{lem: pi restr}
Let $\mcC$ be an extriangulated category with enough
$\mbE$-projectives and $\mbE$-injectives, and $\free (J)=J\subseteq \mcC$
 satisfying $J\subseteq \mcP(\mcC)\cap \mcI(\mcC)$. 
 Let $\mcC_J:=\mcC/[J]$ and 
$\pi_J:\mcC\to \mcC_J$ be the corresponding canonical
projection. Then, the following statements hold for any $n\geq 1$ and any $\mcX\subseteq
\mcC$:
\begin{enumerate}[(a)]
\item The equality
$\pi_J(\mcX^{\perp_{\leq n+1}})=(\pi_J(\mcX))^{\perp_{\leq n+1, J}}$ holds true. In particular, $\pi_J(\mcX^{\perp_{\leq n+1}})$ is an 
extriangulated category.

\item The restriction 
$$\widetilde{\pi_J}: \mcX^{\perp_{\leq n+1}}\to \pi_J(\mcX^{\perp_{\leq n+1}}),\qquad f\mapsto \pi_J(f),$$ is full and essentially surjective.
\end{enumerate}
\end{lemma}

\begin{proof}
(a) Let $\mcX\subseteq \mcC$ and $X\in \mcX$. Consider $M\in \mcC_J$ such
that $M\simeq \pi_J(M'):=M'$ for some $M'\in \mcC$. By 
Proposition~\ref{prop: quot C suf proj}, for any $i\in [1,n]$, we have
$$\mbE_{\mcC_J}^{i}(X, M)\cong 
\mbE_{\mcC_J}^{i}(X, M')\cong
\mbE^{i}(X, M').
$$
Thus, $M\in (\pi_J(\mcX))^{\perp_{\leq n+1, J}}$ if and only if $M\simeq \pi_J(M')$ for some 
$M'\in \mcX^{\perp_{\leq n+1}}$. Finally, since $(\pi_J(\mcX))^{\perp_{\leq n+1, J}}$ is closed under 
extensions in $\mcC_J$, so is 
$\pi_J(\mcX^{\perp_{\leq n+1}})$ and then
it is an extriangulated category 
(see~\cite[Corollary 3.12]{NP19}).

(b) is clear due to $\mcX$ is a full subcategory and $\pi_J$ is
full and essentially surjective.
\end{proof}

\subsection*{Equivalences between extriangulated categories}

We continue this section with a particular class of functors between extriangulated categories. Namely, the so-called \emph{exact functors} firstly appeared in \cite{bennetttransport}. Below we recall this definition from \cite{Ogawa}.

\begin{definition}\cite[Definition 2.11 (1)]{Ogawa}\label{def: iso entre extrian}
Let $(\mcC, \mbE, \mathfrak{s})$ and $(\mcD, \mathbb{F}, \mathfrak{t})$ be extriangulated categories. An \textbf{exact functor} $(G, \phi): (\mcC, \mbE, \mathfrak{s})\to (\mcD, \mathbb{F}, \mathfrak{t})$ is a pair of an additive functor 
$G:\mcC\to \mcD$ and a natural transformation $\phi: \mbE\Rightarrow \mathbb{F}\circ (G^{op}\times G)$ such that 
\[\begin{tikzcd}[column sep=2.25em]
	G(A) & G(B) & G(C) & {{}}
	\arrow["{G(x)}", from=1-1, to=1-2]
	\arrow["{G(y)}", from=1-2, to=1-3]
	\arrow["{\phi_{C, A}(\delta)}", dashed, from=1-3, to=1-4]
\end{tikzcd}\]
is an $\mathbb{F}$-triangle in $\mcD$, for any $\mbE$-triangle $A\mathop{\longrightarrow}\limits^{x} B\mathop{\longrightarrow}\limits^{y} C\mathop{\dashrightarrow}\limits^{\delta}$ in $\mcC$.
\end{definition}

The following outcome allows us to say when an 
exact functor between extriangulated categories is,
indeed, an equivalence of extriangulated categories (for more details, see \cite[Definitions 2.11 (2) \& (3) and Remark 2.12]{Ogawa}).

\begin{proposition}\cite[Proposition 2.13]{NOSlocalization}\label{pro: eq extr}
Let $(\mcC, \mbE, \mathfrak{s})$, $(\mcD, \mathbb{F}, \mathfrak{t})$ be extriangulated categories and
$(G, \phi): (\mcC, \mbE, \mathfrak{s})\to (\mcD, \mathbb{F}, \mathfrak{t})$ be an exact functor. Then, 
$(G, \phi)$ is an
equivalence of extriangulated categories if and only if
$G$ is an equivalence of categories and $\phi$ is a natural 
isomorphism.
\end{proposition}

\begin{remark}\label{rmk: equiv extriang}
Two extriangulated categories $(\mcC, \mbE, \mathfrak{s})$ and $(\mcD, \mathbb{F}, \mathfrak{t})$ are equivalent as extriangulated categories if and only if:
\begin{enumerate}[(a)]
\item There exists an equivalence of additive categories $G:\mcC\to \mcD$, and
\item There exists a natural isomorphism $\phi:\mbE \Rightarrow \mathbb{F}\circ (G^{op}\times G)$ such that
\[\begin{tikzcd}[column sep=2.25em]
	G(A) & G(B) & G(C) & {{}}
	\arrow["{G(x)}", from=1-1, to=1-2]
	\arrow["{G(y)}", from=1-2, to=1-3]
	\arrow["{\phi_{C, A}(\delta)}", dashed, from=1-3, to=1-4]
\end{tikzcd}\]
is an $\mathbb{F}$-triangle in $\mcD$, for any $\mbE$-triangle $A\mathop{\longrightarrow}\limits^{x} B\mathop{\longrightarrow}\limits^{y} C\mathop{\dashrightarrow}\limits^{\delta}$ in $\mcC$.
\end{enumerate}
\end{remark}

We finish giving the main result of this section which consists in proving that certain quotient categories coming from $(n+2)$-rigid categories 
($\mcX\subseteq \mcC$ is $(n+2)$-rigid if 
$\mbE^{j}(\mcX, \mcX)=0$ for all $j\in [1, n+1]$) are equivalent. It is 
worth mentioning that there exists a previous treatment about this 
in \cite{FMP23} where the authors 
introduce \emph{tautological functors} to describe
quotient categories of $2$-rigid and functorially finite subcategories in 
Frobenius extriangulated categories \cite[Definition 4.1 and Theorem 4.16]{FMP23}. In this work 
we address the analogue of this equivalence
in the sense of 
Remark~\ref{rmk: equiv extriang}.

\begin{theorem}\label{teo: equiv extriang}
Let $\mcC$ be an extriangulated category with enough
$\mbE$-projectives and $\mbE$-injectives, and 
$\free(J)=J\subseteq \mcC$ satisfying 
$J\subseteq \mcP(\mcC)\cap \mcI(\mcC)$. 
Let $\mcC_J:=\mcC/[J]$ and 
$\pi_J:\mcC\to \mcC_J$ be the corresponding canonical
projection. If $n\geq 1$ and $\free(\mcX)=\mcX$ is an $(n+2)$-rigid subcategory of $\mcC$
such that $\mcX^{\perp_{\leq n+1}}={}^{\perp_{\leq n+1}}\mcX$ then, there exists an equivalence of extriangulated categories $$(G, \phi):\mcX^{\perp_{\leq n+1}}/ [\add(\mcX\cup J)]\longrightarrow
\pi_J(\mcX^{\perp_{\leq n+1}})/[\add(\pi_J(\mcX))].$$ 
\end{theorem}

\begin{proof} 
Notice first that any object in $\add(\mcX\cup J)$ is projective-injective in $\mcX^{\perp_{\leq n+1}}$ and  any object in $\add (\pi_J(\mcX))$ is projective-injective in $\pi_J(\mcX^{\perp_{\leq n+1}})$ by Proposition~\ref{prop: quot C suf proj}. 
Thus, the corresponding quotient categories
$$\mathcal{L}:=\mcX^{\perp_{\leq n+1}}/ [\add(\mcX\cup J)]\quad \mbox{ and }\quad \mathcal{R}:=\pi_J(\mcX^{\perp_{\leq n+1}})/[\add(\pi_J(\mcX))],$$ 
are extriangulated categories (see Proposition~\ref{estruct del cociente}).
In the 
sequel, we will also denote 
them by $(\mathcal{L}, \mbE_{\mathcal{L}}, \mathfrak{s}_{\mathcal{L}})$ and $(\mathcal{R}, \mbE_{\mathcal{R}}, \mathfrak{s}_{\mathcal{R}})$,
respectively, and we will write $\rho:\pi_J(\mcX^{\perp_{\leq n+1}})\to
\mathcal{R}$ and $\nu: \mcX^{\perp_{\leq n+1}}\to
\mathcal{L}$ to denote the corresponding canonical projections of each one.\\

(1) We first prove that there is an equivalence of
additive categories $G:\mathcal{L}\longrightarrow
\mathcal{R}.$\\

Let us consider the following map
(see Lemma~\ref{lem: pi restr}):
\[
\begin{array}{ll}
\pi: \mcX^{\perp_{\leq n+1}}\longrightarrow
\pi_J(\mcX^{\perp_{\leq n+1}}),\quad & (A\mathop{\longrightarrow}\limits^{f} B)\mapsto
(A\mathop{\longrightarrow}\limits^{\pi_J(f)} B)
\end{array}
\]
Since 
$\add(\mcX\cup J)\subseteq \Ker (\rho\pi)$, by universal property of $\nu$, there is a functor 
$G:\mathcal{L}\longrightarrow \mathcal{R}$ such that
$G\nu=\rho\pi$. That is, we have the following commutative diagram
\[\begin{tikzcd}[row sep=3.15em]
	\mcX^{\perp_{\leq n+1}} && \pi_J(\mcX^{\perp_{\leq n+1}}) \\
	\mathcal{L} && \mathcal{R}
	\arrow["\pi", from=1-1, to=1-3]
	\arrow["\rho", from=1-3, to=2-3]
	\arrow["\nu"', from=1-1, to=2-1]
	\arrow["G"', dashed, from=2-1, to=2-3]
\end{tikzcd}\]
Moreover, $G$
is full and essentially surjective due to $\rho$, $\pi$ and $\nu$ are (see Lemma~\ref{lem: pi restr}).
Thus, to prove that $G$ is an equivalence, it is suffices to show $\Ker(\rho\pi)\subseteq [\add(\mcX\cup J)]$.

Indeed, let $h:M\to N \in \Ker(\rho\pi)$. Then, $(\rho\pi)(h)=0$, \emph{i.e.,} $\pi (h)$ factors through an object 
$\pi (X)$ with $X\in \mcX$. By 
using that
$\pi$ is full, we know that there exist morphisms
$f:M\to X$ and $f':X\to N$ with $X\in \mcX$ such
that the following diagram commutes
\[\begin{tikzcd}
	{\pi (M)} && {\pi(N)} \\
	& {\pi (X)}
	\arrow[""{name=0, anchor=center, inner sep=0}, "{\pi (h)}", from=1-1, to=1-3]
	\arrow["{\pi (f)}"', from=1-1, to=2-2]
	\arrow["{\pi (f')}"', from=2-2, to=1-3]
\end{tikzcd}\]
That is, $\pi (f'f)=\pi (f') \pi (f)=\pi (h)$. From this equality, we have that there exist 
morphisms $g:M\to Q$ and $g':Q\to N$ in 
$\mcC$ with $Q\in J$ such that
the following diagram commutes
\[\begin{tikzcd}
	M && N \\
	& Q
	\arrow["{h-f'f}", from=1-1, to=1-3]
	\arrow["{g}"', from=1-1, to=2-2]
	\arrow["{g'}"', from=2-2, to=1-3]
\end{tikzcd}\]
that is, $h=f'f+g'g$. Thus,
$h$ can be rewritten as 
$h=\left( f'\, g'\right)\tiny{\left(
\begin{array}{c}
f\\
g
\end{array}
\right)}$. 

\[
\xymatrix{
M\ar[rr]^{h}\ar[dr]_{\tiny \left(
\begin{array}{c}
f\\ g
\end{array}
\right)} & & N\\
& X\oplus Q\ar[ur]_{(f'\, g')} &
}
\]
and so
$h\in [\add(\mcX\cup J)]$. Therefore,
$G: \mathcal{L}\longrightarrow \mathcal{R}$
is an equivalence of categories.\\

(2) For any $(C, A)\in \mathcal{L}^{op}\times \mathcal{L}$, we define $\phi_{C, A}:\mbE_\mathcal{L}(C, A)\rightarrow \mbE_{\mathcal{R}}(C, A)$ as the composition of the following identities (see Proposition~\ref{estruct del cociente})
\[\begin{tikzcd}[sep=tiny]
	\mbE_{\mathcal{L}}(C,A) && \mbE_{\mcX^{\perp_{\leq n+1}}}(C,A) && \mbE_{\pi(\mcX^{\perp_{\leq n+1}})}(C,A) && \mbE_{\mathcal{R}}(C,A) \\
	{\nu(\delta)} & \mapsto & \delta & \mapsto & {\pi(\delta)} & \mapsto & {\rho\pi(\delta)}
	\arrow["{\mathrm{1d}}", from=1-1, to=1-3]
	\arrow["{\mathrm{1d}}", from=1-3, to=1-5]
	\arrow["{\mathrm{1d}}", from=1-5, to=1-7]
\end{tikzcd}\]
It is clear that $\phi_{C, A}$ is an isomorphism. Thus, it 
remains to prove that $$\phi: \mbE_{\mathcal{L}}\Rightarrow \mbE_{\mathcal{R}}\circ (G^{op}\times G)$$ given by
 $$\phi:=\{\phi_{C,A}:\mbE_{\mathcal{L}}(C,A)\mathop{\rightarrow} \mbE_{\mathcal{R}}(C,A)
 \}_{(C,A)\in \mathcal{L}^{op}\times \mathcal{L}}$$
 is a natural transformation.\\
 
 Let $f:A\to A'$ and $g:C'\to C$ be morphisms in $\mcX^{\perp_{\leq n+1}}$. We see that the following diagram commutes

\begin{equation}\label{diag iso nat}
\begin{tikzcd}[row sep=3.15em]
	\mbE_{\mathcal{L}}(C, A) &&&& \mbE_{\mathcal{L}}(C', A') \\
	\mbE_{\mathcal{R}}(C, A) &&&& \mbE_{\mathcal{R}}(C', A')
	\arrow["{\mbE_{\mathcal{L}}((\nu (g))^{op}, \nu (f))}", from=1-1, to=1-5]
	\arrow["{\phi_{C, A}}"', from=1-1, to=2-1]
	\arrow["{\phi_{C',A'}}", from=1-5, to=2-5]
	\arrow["{\mbE_{\mathcal{R}}((G\nu (g))^{op}, G\nu (f))}"', from=2-1, to=2-5]
\end{tikzcd}
\end{equation}
$\,$\\

In fact, let $\nu (\delta)\in \mbE_{\mathcal{L}}(C, A)$. On the one hand, from Lemma~\ref{lem: pi y cdot}, we get:

\[
\begin{array}{rl}
[\phi_{C', A'}\circ \mbE_{\mathcal{L}}((\nu (g))^{op}), \nu (f))](\nu (\delta))= & \phi_{C', A'}(\nu (f)\cdot \nu (\delta)\cdot \nu (g))\\
= & \phi_{C', A'}(\nu (f\cdot \delta \cdot g))\\
= & \rho\pi(f\cdot \delta\cdot g).
\end{array}
\]

On the other hand, from Lemma~\ref{lem: pi y cdot} again and
the equality $G\nu=\rho\pi$ we have:
\[
\begin{array}{rl}
[\mbE_{\mathcal{R}}((G\nu (g))^{op}, G\nu (f))\circ 
\phi_{C, A}](\nu (\delta)) = & \mbE_{\mathcal{R}}(\rho\pi(g)^{op}, \rho\pi(f))(\phi_{C, A}(\nu (\delta)))\\
= & \mbE_{\mathcal{R}}(\rho\pi(g)^{op}, \rho\pi(f))(\rho\pi(\delta))\\
= & \rho\pi (f) \cdot \rho\pi (\delta)\cdot \rho\pi (g)\\
= & \rho\pi(f\cdot \delta\cdot g).
\end{array}
\]
Therefore~\eqref{diag iso nat}
commutes.\\

(3) Finally, we see that for any $\nu (\delta)\in \mbE_\mathcal{L}(C, A)$ with
$A\mathop{\longrightarrow}\limits^{\nu (x)} B\mathop{\longrightarrow}\limits^{\nu (y)} C\mathop{\dashrightarrow}\limits^{\nu (\delta)}$, we have that
\[\begin{tikzcd}[column sep=2.5em]
	A & B & C & {{}}
	\arrow["{G\nu(x)}", from=1-1, to=1-2]
	\arrow["{G\nu(y)}", from=1-2, to=1-3]
	\arrow["{\phi_{C, A}(\nu(\delta))}", dashed, from=1-3, to=1-4]
\end{tikzcd}\]
is an $\mbE_\mathcal{R}$-triangle in $\mathcal{R}$.\\

Indeed, let $\nu (\delta)\in \mbE_\mathcal{L}(C, A)$. Since $\nu(\delta)$ is identified with $\delta\in \mbE(C, A)$ one can consider its realization in $\mcC$, we say $\mathfrak{s}(\delta)=[A\mathop{\longrightarrow}\limits^{x} B\mathop{\longrightarrow}\limits^{y} C].$ Thus, from Proposition~\ref{estruct del cociente}, it follows 
that $\mathfrak{s}_\mathcal{L}(\nu (\delta))=[A\mathop{\longrightarrow}\limits^{\nu (x)} B\mathop{\longrightarrow}\limits^{\nu (y)} C]$ 
is the corresponding realization of $\nu (\delta)$ in $\mathcal{L}$ and $\mathfrak{s}_{\pi(\mcX^{\perp_{\leq n+1}})} (\pi(\delta))=[A\mathop{\longrightarrow}\limits^{\pi (x)} B\mathop{\longrightarrow}\limits^{\pi (y)} C]$
is the corresponding realization of $\pi(\delta)$
in $\pi(\mcX^{\perp_{\leq n+1}}).$

Thus, from Proposition~\ref{estruct del cociente} again, we also have
$$[A\mathop{\longrightarrow}\limits^{\rho\pi(x)} B\mathop{\longrightarrow}\limits^{\rho\pi(y)} C]=\mathfrak{s}_\mathcal{R}(\rho\pi(\delta))=\mathfrak{s}_\mathcal{R}(\phi_{C, A}(\nu (\delta))).$$

So, by using that $G\nu=\rho\pi$, we get the equality
$$\mathfrak{s}_\mathcal{R}(\phi_{C, A}(\nu (\delta)))=
[A\mathop{\longrightarrow}\limits^{G\nu(x)} B\mathop{\longrightarrow}\limits^{G\nu(y)} C]$$
holds true. Therefore, we can conclude $(G, \phi):\mathcal{L}\to \mathcal{R}$ is an equivalence of extriangulated categories by Remark~\ref{rmk: equiv extriang}.
\end{proof}

\section{\textbf{Examples coming from reduction process}}
\label{sec: examples}

In this section we provide several examples of the 
equivalence given in Theorem~\ref{teo: equiv extriang}. To do that, it is quite natural to think 
under which settings the equality $\mcX^{\perp_{\leq n+1}}={}^{\perp_{\leq n+1}}\mcX$ holds true for
$(n+2)$-rigid subcategories. A partial answer can be
gotten through the \emph{reduction process}.

The reduction process for rigid subcategories in extriangulated categories has been addressed for
example in \cite{IYmutation, FMP23}. In these works, it was proven that both quotients in Theorem~\ref{teo: equiv extriang} are triangulated and, indeed, there exists an equivalence of triangulated categories. Recently, this concept was extended in \cite{huerta2024reduction} for $(n+2)$-rigid and functorially finite subcategories (without the condition of being Frobenius).
In this section, we will work with the quotient category 
determined by the ideal of projective-injective objects in such reduction in order to get new examples. We begin recalling the 
following definition.

\begin{definition}\cite[Definition 3.1]{huerta2024reduction}\label{def: mi reduction}
    Let $\mcC$ be an extriangulated category with enough
    $\mbE$-projectives and $\mbE$-injectives, and $\mcX\subseteq \mcC$ be $(n+2)$-rigid and functorially finite in $\mcC$ such that $\mcX^{\perp_{\leq n+1}}={}^{\perp_{\leq n+1}}\mcX$.
    The reduction of $\mcC$ at $\mcX$, denoted by $\mathcal{R}_\mcC^{n+1}(\mcX)$, is defined as
    $$\mathcal{R}_\mcC^{n+1}(\mcX):=\mcX^{\perp_{\leq n+1}}={}^{\perp_{\leq n+1}}\mcX.$$
\end{definition}

In 
\cite[Corollary 3.3]{huerta2024reduction}
 it is proven that the class of projective-injective objects in this new category coincides with
$$\mathcal{I}:=\add(\mcX\cup \mathcal{P}(\mcC))\cap \add(\mcX\cup \mathcal{I}(\mcC))$$
% where the functorially finiteness condition plays an significant role for such description.
and, from \cite [Proposition 3.30]{NP19}, the quotient $\mathcal{R}_\mcC^{n+1}(\mcX)/\mathcal{I}$ is an
extriangulated category.

On the other hand, by applying  Theorem~\ref{teo: equiv extriang}, we get another extriangulated category, namely, 
$\mathcal{R}_\mcC^{n+1}(\mcX)/[\add(\mcX\cup J)]$ where
$J:=\mathcal{P}(\mcC)\cap \mathcal{I}(\mcC)$. Notice that $$\add(\mcX\cup J)\subseteq \add(\mcX\cup \mcP(\mcC))\cap \add(\mcX\cup \mcI(\mcC))$$ and there are some cases where the equality holds (for example, when
$\mcP(\mcC)=\mcI(\mcC)$). The following lemma shows that the equality remains valid when $\mcC$ is  Krull-Schmidt.\\

Let $\mcC$ be an additive category. We recall that $\mcC$ is {\bf Krull-Schmidt} if each object in $\mcC$ decomposes into a finite coproduct of objects having local endomorphisms ring. In
particular, every summand of this decomposition is an indecomposable object in
$\mcC$. We denote by $\mathrm{ind}(\mcC)$ the full subcategory of $\mcC$ whose objects are determined by choosing one object for each iso-class of indecomposable objects in $\mcC$ (for more details, see
\cite{HK15}). 

\begin{lemma}\label{lem: KS}
Let $\mcC$ be a Krull-Schmidt extriangulated category with 
enough $\mbE$-projectives and $\mbE$-injectives,
$\add(\mcY)=\mcY\subseteq \mcC$ and $J:=\mcP(\mcC)\cap \mcI(\mcC)$. Then, the equality 
$$\add(\mcY\cup J)=
\add(\mcY\cup \mcP(\mcC))\cap \add(\mcY\cup \mcI(\mcC))$$
holds true.
\end{lemma}

\begin{proof}
On the one hand, since $M\in \add(\mcY\cup \mcP(\mcC))\cap \add(\mcY \cup 
\mcI(\mcC))$, there exist $Z, Z'\in \mcC$ such that
$M\oplus Z=Y\oplus P$ and $M\oplus Z'=Y'\oplus I$ where
$Y, Y'\in \mcY$, $P\in \mcP(\mcC)$ and $I\in \mcI(\mcC)$. Thus, since $\mcC$ is Krull-Schmidt, there
are decompositions
$$M=\bigoplus_{i=1}^{n}M_i,\, Y=\bigoplus_{j=1}^{m}Y_j,\,
Y'=\bigoplus_{k=1}^{t}Y'_k,\, P=\bigoplus_{l=1}^{r}P_l\,
\mbox{ and }\, I=\bigoplus_{u=1}^{s}I_u$$
with
$M_i, Y_j, Y'_k, 
P_l, I_u\in \mathrm{ind} (\mcC)$ and $\forall\, i,j,k,l,u$. Now, since $$M_i\mid Y\oplus P=\left(\bigoplus_{j=1}^{m}Y_j\right)\oplus \left( 
\bigoplus_{l=1}^{r}P_l\right)$$ we have that $M_i\in \mcY$
or $M_i\in \mcP(\mcC)$. Dually, by using that
$$M_i\mid Y'\oplus I=\left(\bigoplus_{k=1}^{t}Y'_k\right)\oplus 
\left(\bigoplus_{u=1}^{s}I_u\right)$$
we get $M_i\in \mcY$ or $M_i\in \mcI(\mcC)$.
Then, $M_i\in \mcY$ or $M_i\in J$ for every
$i\in [1,n]$ and so
$M\in \add(\mcY\cup J)$. Finally, the containment $\subseteq$ is clear and hence,
we get the equality.
\end{proof}

We finish this section by computing $\mathcal{R}_\mcC^{n+1}(\mcX)/[\add(\mcX\cup J)]$ in several examples on Krull-Schmidt extriangulated categories. To do that, the following 
result from \cite{HMSS2023quotient} will be used.

\begin{proposition}\cite[Proposition 2.9]{HMSS2023quotient}\label{pro: KS}
    For an additive category $\mcC$, $\mcX\subseteq \mcY\subseteq \mcC$ such that 
    $\mcX=\add (\mcX)$ and $\mcY=\free (\mcY)$ and the quotient category $\mcD:=\mcY/\mcX$, if $\mcY$ is Krull-Schmidt then $\mcD$ is Krull-Schmidt and 
    $\ind (\mcD)=\ind (\mcY)\setminus \mcX$.
\end{proposition}

In the examples below, recall that an extriangulated 
category $\mcC$ is \emph{Frobenius} if $\mcC$ has 
enough $\mbE$-projectives and $\mbE$-injectives, 
and $\mathcal{P}(\mcC)=\mathcal{I}(\mcC)$.

\begin{example}\label{ex3} Let 
$k$ be an algebraically closed field and $\Lambda$ be the path algebra over $k$ given by the quiver
\[\tiny{\begin{tikzcd}
	\bullet && \bullet && \bullet && \bullet && \bullet && \bullet
	\arrow["x", from=1-1, to=1-3]
	\arrow["x", from=1-3, to=1-5]
	\arrow["x", from=1-5, to=1-7]
	\arrow["x", from=1-7, to=1-9]
	\arrow["x", from=1-9, to=1-11]
\end{tikzcd}}\]
with relation $x^{3}=0$. Then, the Auslander-Reiten quiver of $\modu(\Lambda)$ is given by 
the following diagram
\[\begin{tikzcd}[sep=small]
	&& \blacklozenge && \blacklozenge && \bullet && \bullet \\
	& \diamondsuit && \clubsuit && \circ && \bullet && \bullet \\
	\diamondsuit && \clubsuit && \clubsuit && \circ && \bullet && \heartsuit
	\arrow[from=3-1, to=2-2]
	\arrow[from=2-2, to=1-3]
	\arrow[from=1-3, to=2-4]
	\arrow[from=2-4, to=3-5]
	\arrow[from=3-3, to=2-4]
	\arrow[from=2-4, to=1-5]
	\arrow[from=1-5, to=2-6]
	\arrow[from=2-6, to=3-7]
	\arrow[from=3-5, to=2-6]
	\arrow[from=2-6, to=1-7]
	\arrow[from=1-7, to=2-8]
	\arrow[from=2-8, to=3-9]
	\arrow[from=3-7, to=2-8]
	\arrow[from=2-8, to=1-9]
	\arrow[from=1-9, to=2-10]
	\arrow[from=2-10, to=3-11]
	\arrow[from=3-9, to=2-10]
	\arrow[from=2-2, to=3-3]
	\arrow[dotted, no head, from=2-2, to=2-4]
	\arrow[dotted, no head, from=2-4, to=2-6]
	\arrow[dotted, no head, from=2-6, to=2-8]
	\arrow[dotted, no head, from=2-8, to=2-10]
	\arrow[dotted, no head, from=3-1, to=3-3]
	\arrow[dotted, no head, from=3-3, to=3-5]
	\arrow[dotted, no head, from=3-5, to=3-7]
	\arrow[dotted, no head, from=3-7, to=3-9]
	\arrow[dotted, no head, from=3-9, to=3-11]
\end{tikzcd}\]
Consider the following subcategories of $\modu (\Lambda)$:
\begin{enumerate}[(a)]
\item $\mcX=\add(\mcX)$ whose indecomposable objects are denoted by $\diamondsuit$, $\blacklozenge$ and $\circ$;
\item $\mcB_1=\add(\mcB_1)$ whose indecomposable objects are denoted by $\diamondsuit$, $\blacklozenge$, $\circ$ and $\clubsuit$;
\item $\mcB_2=\add(\mcB_2)$ whose indecomposable objects are denoted by $\heartsuit$; and
\item $\mcC:=\add(\mcB_1\cup \mcB_2)$.
\end{enumerate}
Since $\mcC$ is closed under extensions in 
$\modu (\Lambda)$, 
$\mcC$ admits an extriangulated structure 
(see~\cite[Corollary 3.12]{NP19}). Moreover, this is given by:\footnote{We use another extriangulated structure in comparison with the
given one in \cite[Example 4.2]{LZ20arising}.}

\begin{enumerate}
\item[$\bullet$] $\mbE(X, Y)=\Ext^{1}_{\modu (\Lambda)}(X, Y)$, for any $X, Y\in \mcB_1$.
\item[$\bullet$] $\mbE(X, Y)=0=\mbE(Y, X)$, for any $X\in \mcB_1$ and $Y\in \mcB_2$.
\item[$\bullet$] $\mbE(X, Y)=0$, for any $X, Y\in \mcB_2$.
\end{enumerate}

On the other hand, $\mcC$ has enough $\mbE$-projectives 
and $\mbE$-injectives where 
$$\mcP(\mcC)=\add(
\diamondsuit \cup \blacklozenge \cup \heartsuit) \quad \mbox{ and }\quad \mcI(\mcC)=\add(
\blacklozenge \cup \circ \cup \heartsuit).$$
Thus, $\mcC$ is not Frobenius.

On the other hand, concerning $\mcX$, we know
that: $\mcX$ is $2$-rigid, functorially finite in $\mcC$ and the equality  
$\mcX^{\perp_1}={}^{\perp_1}\mcX=\add(\mcX\cup \heartsuit)$ holds true. Furthermore, since
$$\mathcal{P}(\mcC)\subseteq \add(\mcX\cup \mathcal{I}(\mcC))\quad  \mbox{ and }\quad  
\mathcal{I}(\mcC)\subseteq \add(\mcX\cup \mathcal{P}(\mcC))$$ it follows by \cite[Corollary 3.9]{huerta2024reduction} that $\mathcal{R}^{1}_\mcC(\mcX):=\mcX^{\perp_1}$ is Frobenius. Thus, 
from Lemma~\ref{lem: KS},
$J':=\add(\mcX\cup \heartsuit)$ and so by Theorem~\ref{teo: equiv extriang} and Proposition~\ref{pro: KS} we get an equivalence of extriangulated categories $$\mathcal{R}^{1}_\mcC(\mcX)/[J']\cong \pi(\mcX^{\perp_1})/[\add(\pi(\mcX))]=\{0\}.$$
\end{example}

\begin{example}\label{ex4}
Let $\Lambda$ be the self-injective Nakayama algebra associated to the following quiver
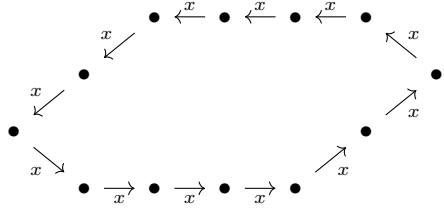
\[\begin{tikzcd}[sep=small]
	&& \bullet & \bullet & \bullet & \bullet \\
	& \bullet &&&&& \bullet \\
	\bullet &&&&& \bullet \\
	& \bullet & \bullet & \bullet & \bullet
	\arrow["x"', from=1-6, to=1-5]
	\arrow["x"', from=1-5, to=1-4]
	\arrow["x"', from=1-4, to=1-3]
	\arrow["x"', from=1-3, to=2-2]
	\arrow["x"', from=2-2, to=3-1]
	\arrow["x"', from=3-1, to=4-2]
	\arrow["x"', from=4-2, to=4-3]
	\arrow["x"', from=4-3, to=4-4]
	\arrow["x"', from=4-4, to=4-5]
	\arrow["x"', from=4-5, to=3-6]
	\arrow["x"', from=3-6, to=2-7]
	\arrow["x"', from=2-7, to=1-6]
\end{tikzcd}\]
with relation $x^{4}=0$. Then, the Auslander-Reiten quiver of the stable category of 
$\modu (\Lambda)$, denoted by $\underline{\modu}(\Lambda)$, is described as follows\\

\adjustbox{scale=0.64,center}{
\begin{tikzcd}[column sep=tiny,row sep=huge]
	\times && \clubsuit && \bullet && \bullet && \bullet && \bullet && \bullet && \spadesuit && \times && \times && \times && \times && \times \\
	{{}} & \clubsuit && \bullet && \bullet && \bullet && \bullet && \bullet && \bullet && \spadesuit && \times && \times && \times && \times & {{}} \\
	\clubsuit && \bullet && \bullet && \bullet && \bullet && \bullet && \bullet && \bullet && \spadesuit && \times && \bigstar && \times && \clubsuit
	\arrow[dashed, no head, from=2-1, to=2-2]
	\arrow[dashed, no head, from=2-24, to=2-25]
	\arrow[dashed, no head, from=1-1, to=1-3]
	\arrow[dashed, no head, from=1-3, to=1-5]
	\arrow[dashed, no head, from=1-5, to=1-7]
	\arrow[dashed, no head, from=1-7, to=1-9]
	\arrow[dashed, no head, from=1-9, to=1-11]
	\arrow[dashed, no head, from=1-11, to=1-13]
	\arrow[dashed, no head, from=1-13, to=1-15]
	\arrow[dashed, no head, from=1-15, to=1-17]
	\arrow[dashed, no head, from=1-17, to=1-19]
	\arrow[dashed, no head, from=1-19, to=1-21]
	\arrow[dashed, no head, from=1-21, to=1-23]
	\arrow[dashed, no head, from=1-23, to=1-25]
	\arrow[dashed, no head, from=2-24, to=2-22]
	\arrow[dashed, no head, from=3-25, to=3-23]
	\arrow[dashed, no head, from=2-22, to=2-20]
	\arrow[dashed, no head, from=3-23, to=3-21]
	\arrow[dashed, no head, from=2-20, to=2-18]
	\arrow[dashed, no head, from=3-21, to=3-19]
	\arrow[dashed, no head, from=2-18, to=2-16]
	\arrow[dashed, no head, from=3-19, to=3-17]
	\arrow[dashed, no head, from=2-16, to=2-14]
	\arrow[dashed, no head, from=3-17, to=3-15]
	\arrow[dashed, no head, from=3-15, to=3-13]
	\arrow[dashed, no head, from=3-13, to=3-11]
	\arrow[dashed, no head, from=2-14, to=2-12]
	\arrow[dashed, no head, from=2-12, to=2-10]
	\arrow[dashed, no head, from=2-2, to=2-4]
	\arrow[dashed, no head, from=2-4, to=2-6]
	\arrow[dashed, no head, from=2-6, to=2-8]
	\arrow[dashed, no head, from=2-8, to=2-10]
	\arrow[dashed, no head, from=3-1, to=3-3]
	\arrow[dashed, no head, from=3-3, to=3-5]
	\arrow[dashed, no head, from=3-5, to=3-7]
	\arrow[dashed, no head, from=3-7, to=3-9]
	\arrow[dashed, no head, from=3-9, to=3-11]
	\arrow[from=1-1, to=2-2]
	\arrow[from=2-2, to=3-3]
	\arrow[from=3-1, to=2-2]
	\arrow[from=2-2, to=1-3]
	\arrow[from=3-3, to=2-4]
	\arrow[from=2-4, to=1-5]
	\arrow[from=1-3, to=2-4]
	\arrow[from=2-4, to=3-5]
	\arrow[from=3-5, to=2-6]
	\arrow[from=1-5, to=2-6]
	\arrow[from=2-6, to=3-7]
	\arrow[from=2-6, to=1-7]
	\arrow[from=1-7, to=2-8]
	\arrow[from=2-8, to=3-9]
	\arrow[from=3-7, to=2-8]
	\arrow[from=2-8, to=1-9]
	\arrow[from=1-9, to=2-10]
	\arrow[from=2-10, to=3-11]
	\arrow[from=3-9, to=2-10]
	\arrow[from=2-10, to=1-11]
	\arrow[from=1-11, to=2-12]
	\arrow[from=2-12, to=3-13]
	\arrow[from=3-11, to=2-12]
	\arrow[from=2-12, to=1-13]
	\arrow[from=1-13, to=2-14]
	\arrow[from=2-14, to=3-15]
	\arrow[from=3-13, to=2-14]
	\arrow[from=2-14, to=1-15]
	\arrow[from=1-15, to=2-16]
	\arrow[from=3-15, to=2-16]
	\arrow[from=2-16, to=3-17]
	\arrow[from=2-16, to=1-17]
	\arrow[from=1-17, to=2-18]
	\arrow[from=2-18, to=3-19]
	\arrow[from=3-17, to=2-18]
	\arrow[from=2-18, to=1-19]
	\arrow[from=1-19, to=2-20]
	\arrow[from=2-20, to=3-21]
	\arrow[from=3-19, to=2-20]
	\arrow[from=2-20, to=1-21]
	\arrow[from=1-21, to=2-22]
	\arrow[from=2-22, to=3-23]
	\arrow[from=3-21, to=2-22]
	\arrow[from=2-22, to=1-23]
	\arrow[from=1-23, to=2-24]
	\arrow[from=2-24, to=3-25]
	\arrow[from=3-23, to=2-24]
	\arrow[from=2-24, to=1-25]
\end{tikzcd}
}
$\,$\\
where the first and the last column are identified. Moreover, $\underline{\modu}(\Lambda)$ is triangulated by \cite[Corollary 7.4]{NP19}.\\ 

Consider the following subcategories of 
$\underline{\modu}(\Lambda)$:

\begin{enumerate}[(a)]
\item $\mcX=\add(\mcX)$ whose indecomposable objects are marked by $\clubsuit$ and $\spadesuit$; 
\item $\mcB_1=\add(\mcB_1)$ whose indecomposable objects are marked by $\clubsuit, \bullet$ and
$\spadesuit$;
\item $\mcB_2=\add(\mcB_2)$ whose indecomposable objects are marked by $\bigstar$; and
\item $\mcC=\add(\mcB_1\cup \mcB_2)$.
\end{enumerate}

Since $\mcC$ is closed under extensions, $\mcC$ has an extriangulated structure from~\cite[Corollary 3.12]{NP19}. Furthermore, this structure is given by:

\begin{enumerate}
\item[$\bullet$] $\mbE(X, Y)=\Hom_{\underline{\modu}(\Lambda)}(X,Y[1])$, for any $X, Y\in \mcB_1$.
\item[$\bullet$] $\mbE(X, Y)=0=\mbE(Y, X)$, for any $X\in \mcB_1$ and
$Y\in \mcB_2$.
\item[$\bullet$] $\mbE(X, Y)=0$, for any $X, Y\in \mcB_2$.
\end{enumerate} 

In addition, $\mcC$ has enough $\mbE$-projectives and $\mbE$-injectives where
$$\mcP(\mcC)=\add(\clubsuit \cup \bigstar)
\quad \mbox{ and }\quad \mcI(\mcC)=\add(\spadesuit\cup \bigstar).$$
Thus, $\mcC$ is not Frobenius.

Now, concerning $\mcX$, we have that:
$\mcX$ is $4$-rigid, functorially finite in $\mcC$ and the equality $\mcX^{\perp_{\leq 3}}={}^{\perp_{\leq 3}}\mcX=\add(\mcX\cup \bigstar)$ holds true. Therefore, the reduction of $\mcC$ at $\mcX$ is $$\mathcal{R}^{3}_\mcC(\mcX)=\add(\mcX\cup \bigstar).$$ 

On the other hand,
since $\mathcal{P}(\mcC)\subseteq \add(\mcX\cup \mathcal{I}(\mcC))$ and
$\mathcal{I}(\mcC)\subseteq \add (\mcX\cup \mathcal{P}(\mcC))$ we have that 
$\mathcal{R}_\mcC^{3}(\mcX)$ is Frobenius 
(see \cite[Corollary 3.9]{huerta2024reduction}).
So, by using Lemma~\ref{lem: KS}, we have that $J'=\add(\mcX\cup \bigstar)$ and then, by Theorem~\ref{teo: equiv extriang} and Proposition~\ref{pro: KS}, there is
an equivalence of extriangulated categories
$$\mathcal{R}_\mcC^{3}(\mcX)/[J']=\pi(\mcX^{\perp_{\leq 3}})/[\pi(\add(\mcX))]=\{0\}.$$
\end{example}

\begin{example}\label{ex2}
\cite[Example 4.1]{LZ20arising} Let $\Lambda$ be the self-injective Nakayama algebra associated to the following quiver 
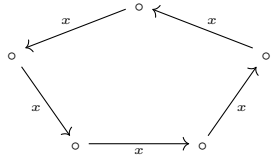
\[\tiny{\begin{tikzcd}
	&& \circ \\
	\circ &&&& \circ \\
	& {} && {} \\
	& \circ && \circ
	\arrow["x"', from=2-1, to=4-2]
	\arrow["x"', from=4-2, to=4-4]
	\arrow["x"', from=4-4, to=2-5]
	\arrow["x"', from=2-5, to=1-3]
	\arrow["x"', from=1-3, to=2-1]
\end{tikzcd}}\]
with relation $x^{3}=0$. Thus, the Auslander-Reiten quiver of
$\mcC:=\modu(\Lambda)$ is given by
\[\tiny{\begin{tikzcd}
	\bullet && \bullet && \bullet && \bullet && \bullet && \bullet \\
{} & \bullet && \bullet && \bullet && \bullet && \bullet & {} \\
	\bullet && \bullet && \bullet && \bullet && \bullet && \bullet
	\arrow[from=1-1, to=2-2]
	\arrow[from=2-2, to=3-3]
	\arrow[from=3-1, to=2-2]
	\arrow[from=2-2, to=1-3]
	\arrow[from=1-3, to=2-4]
	\arrow[from=2-4, to=3-5]
	\arrow[from=3-3, to=2-4]
	\arrow[from=2-4, to=1-5]
	\arrow[from=1-5, to=2-6]
	\arrow[from=2-6, to=3-7]
	\arrow[from=3-5, to=2-6]
	\arrow[from=2-6, to=1-7]
	\arrow[from=1-7, to=2-8]
	\arrow[from=2-8, to=3-9]
	\arrow[from=3-7, to=2-8]
	\arrow[from=2-8, to=1-9]
	\arrow[from=1-9, to=2-10]
	\arrow[from=2-10, to=3-11]
	\arrow[from=3-9, to=2-10]
	\arrow[from=2-10, to=1-11]
	\arrow[dotted, no head, from=3-1, to=3-3]
	\arrow[dotted, no head, from=3-3, to=3-5]
	\arrow[dotted, no head, from=3-5, to=3-7]
	\arrow[dotted, no head, from=3-7, to=3-9]
	\arrow[dotted, no head, from=3-9, to=3-11]
	\arrow[dotted, no head, from=2-2, to=2-4]
	\arrow[dotted, no head, from=2-4, to=2-6]
	\arrow[dotted, no head, from=2-6, to=2-8]
	\arrow[dotted, no head, from=2-8, to=2-10]
	\arrow[dotted, no head, from=2-1, to=2-2]
	\arrow[dotted, no head, from=2-10, to=2-11]
\end{tikzcd}}\]
where the first and last column are identified. Moreover, $\mcC$ is Frobenius.\\

Let $\add(\mcX)=\mcX\subseteq \mcC$ be the
subcategory closed under extensions of $\mcC$
whose indecomposable objects are given by 
$\circ$ in the following diagram 
\[\ind (\mcX):=\tiny{\begin{tikzcd}
	\circ && \circ && \circ && \circ && \circ && \circ \\
{} & \bullet && \bullet && \bullet && \circ && \bullet & {} \\
	\bullet && \circ && \bullet && \bullet && \bullet && \bullet
	\arrow[from=1-1, to=2-2]
	\arrow[from=2-2, to=3-3]
	\arrow[from=3-1, to=2-2]
	\arrow[from=2-2, to=1-3]
	\arrow[from=1-3, to=2-4]
	\arrow[from=2-4, to=3-5]
	\arrow[from=3-3, to=2-4]
	\arrow[from=2-4, to=1-5]
	\arrow[from=1-5, to=2-6]
	\arrow[from=2-6, to=3-7]
	\arrow[from=3-5, to=2-6]
	\arrow[from=2-6, to=1-7]
	\arrow[from=1-7, to=2-8]
	\arrow[from=2-8, to=3-9]
	\arrow[from=3-7, to=2-8]
	\arrow[from=2-8, to=1-9]
	\arrow[from=1-9, to=2-10]
	\arrow[from=2-10, to=3-11]
	\arrow[from=3-9, to=2-10]
	\arrow[from=2-10, to=1-11]
	\arrow[dotted, no head, from=3-1, to=3-3]
	\arrow[dotted, no head, from=3-3, to=3-5]
	\arrow[dotted, no head, from=3-5, to=3-7]
	\arrow[dotted, no head, from=3-7, to=3-9]
	\arrow[dotted, no head, from=3-9, to=3-11]
	\arrow[dotted, no head, from=2-2, to=2-4]
	\arrow[dotted, no head, from=2-4, to=2-6]
	\arrow[dotted, no head, from=2-6, to=2-8]
	\arrow[dotted, no head, from=2-8, to=2-10]
	\arrow[dotted, no head, from=2-1, to=2-2]
	\arrow[dotted, no head, from=2-10, to=2-11]
\end{tikzcd}}\]
In this case, we have that $\mcX$ is $2$-rigid and functorially
finite in $\mcC$ and the equality 
$\mcX^{\perp_1}={}^{\perp_1}\mcX$ holds true.
Moreover, the indecomposable objects in $\mcX^{\perp_1}$ are marked by $\circ$ in the 
following diagram   
\[\ind (\mcX^{\perp_1}):=\tiny{\begin{tikzcd}
	\circ && \circ && \circ && \circ && \circ && \circ \\
{} & \circ && \circ && \bullet && \circ && \bullet & {} \\
	\bullet && \circ && \bullet && \circ && \circ && \bullet
	\arrow[from=1-1, to=2-2]
	\arrow[from=2-2, to=3-3]
	\arrow[from=3-1, to=2-2]
	\arrow[from=2-2, to=1-3]
	\arrow[from=1-3, to=2-4]
	\arrow[from=2-4, to=3-5]
	\arrow[from=3-3, to=2-4]
	\arrow[from=2-4, to=1-5]
	\arrow[from=1-5, to=2-6]
	\arrow[from=2-6, to=3-7]
	\arrow[from=3-5, to=2-6]
	\arrow[from=2-6, to=1-7]
	\arrow[from=1-7, to=2-8]
	\arrow[from=2-8, to=3-9]
	\arrow[from=3-7, to=2-8]
	\arrow[from=2-8, to=1-9]
	\arrow[from=1-9, to=2-10]
	\arrow[from=2-10, to=3-11]
	\arrow[from=3-9, to=2-10]
	\arrow[from=2-10, to=1-11]
	\arrow[dotted, no head, from=3-1, to=3-3]
	\arrow[dotted, no head, from=3-3, to=3-5]
	\arrow[dotted, no head, from=3-5, to=3-7]
	\arrow[dotted, no head, from=3-7, to=3-9]
	\arrow[dotted, no head, from=3-9, to=3-11]
	\arrow[dotted, no head, from=2-2, to=2-4]
	\arrow[dotted, no head, from=2-4, to=2-6]
	\arrow[dotted, no head, from=2-6, to=2-8]
	\arrow[dotted, no head, from=2-8, to=2-10]
	\arrow[dotted, no head, from=2-1, to=2-2]
	\arrow[dotted, no head, from=2-10, to=2-11]
\end{tikzcd}}\]
Thus, $\mathcal{R}_\mcC^{1}(\mcX)=\mcX^{\perp_1}$.

On the other hand, since $\mathcal{P}(\mcC)=\mathcal{I}(\mcC)\subseteq \mcX$, from
\cite[Corollary 3.9]{huerta2024reduction}, we 
know that $\mathcal{R}_\mcC^{1}(\mcX)$ is Frobenius and $J'=\add(\mcX\cup \mathcal{P}(\mcC))=\add(\mcX)$ by Lemma~\ref{lem: KS}. Therefore, from Theorem~\ref{teo: equiv extriang} and Proposition~\ref{pro: KS}, there is an
equivalence of extriangulated categories 
$\mathcal{R}_\mcC^{1}(\mcX)/[J']\cong \pi(\mcX^{\perp_1})/[\add(\pi(\mcX))]$ whose indecomposable objects 
are marked by $\circ$ below
\[\ind (\mathcal{R}_\mcC^{1}(\mcX)/[J'])=\tiny{\begin{tikzcd}
	\bullet && \bullet && \bullet && \bullet && \bullet && \bullet \\
{} & \circ && \circ && \bullet && \bullet && \bullet & {} \\
	\bullet && \bullet && \bullet && \circ && \circ && \bullet
	\arrow[from=1-1, to=2-2]
	\arrow[from=2-2, to=3-3]
	\arrow[from=3-1, to=2-2]
	\arrow[from=2-2, to=1-3]
	\arrow[from=1-3, to=2-4]
	\arrow[from=2-4, to=3-5]
	\arrow[from=3-3, to=2-4]
	\arrow[from=2-4, to=1-5]
	\arrow[from=1-5, to=2-6]
	\arrow[from=2-6, to=3-7]
	\arrow[from=3-5, to=2-6]
	\arrow[from=2-6, to=1-7]
	\arrow[from=1-7, to=2-8]
	\arrow[from=2-8, to=3-9]
	\arrow[from=3-7, to=2-8]
	\arrow[from=2-8, to=1-9]
	\arrow[from=1-9, to=2-10]
	\arrow[from=2-10, to=3-11]
	\arrow[from=3-9, to=2-10]
	\arrow[from=2-10, to=1-11]
	\arrow[dotted, no head, from=3-1, to=3-3]
	\arrow[dotted, no head, from=3-3, to=3-5]
	\arrow[dotted, no head, from=3-5, to=3-7]
	\arrow[dotted, no head, from=3-7, to=3-9]
	\arrow[dotted, no head, from=3-9, to=3-11]
	\arrow[dotted, no head, from=2-2, to=2-4]
	\arrow[dotted, no head, from=2-4, to=2-6]
	\arrow[dotted, no head, from=2-6, to=2-8]
	\arrow[dotted, no head, from=2-8, to=2-10]
	\arrow[dotted, no head, from=2-1, to=2-2]
	\arrow[dotted, no head, from=2-10, to=2-11]
\end{tikzcd}}\]
Furthermore, $\mcX^{\perp_1}/[J']$ is triangulated (see \cite[Corollary 7.4]{NP19}).
\end{example}

\section{\textbf{On the compatibility with triangulated categories}}\label{sec: compatibility with triangulated}

As we can see in Example~\ref{ex2} from a Frobenius reduction it is possible to get a triangulated category (see  \cite[Corollary 7.4]{NP19}). Indeed, it follows from a more general fact which involves 
the concept of reduction. In \cite{FMP23}, it was shown that the reduction of $\mcC$ at $\mcX$, where $\mcX$ is a rigid and functorially finite subcategory of $\mcC$, is Frobenius when $\mcC$ is a  Frobenius extriangulated category. From this and  
Definition~\ref{def: mi reduction}, a natural question arises: 
is the quotient category determined by class of its projective-injective objects in $\mathcal{R}_\mcC^{n+1}(\mcX)$ triangulated in the general case (non Frobenius)? Below we present an example which gives a negative answer. To do that, we begin by recalling 
the following definition.

\begin{definition}\cite[Definition 3.2]{msapato2024characterization}
Let $(\mcC, \mbE, \mathfrak{s})$ be an 
extriangulated category and suppose that $\mcC$ 
admits a triangulated structure $(\mcC, [1], \Delta)$. We say that this triangulated 
structure is \emph{$\mbE$-compatible} if for
each distinguished triangle $X\mathop{\longrightarrow}\limits^{u} Y\mathop{\longrightarrow}\limits^{v} Z\to X[1]$
we have that $X\mathop{\longrightarrow}\limits^{u} Y\mathop{\longrightarrow}\limits^{v} Z\mathop{\dashrightarrow}\limits^{\delta}$ is an $\mbE$-triangle 
for some $\delta\in \mbE(Z, X)$.
\end{definition}

In \cite{msapato2024characterization} is given an answer of when an extriangulated category admits an 
$\mbE$-compatible triangulated structure through the following characterization.

\begin{theorem}\cite[Theorem 3.3]{msapato2024characterization}\label{teo: E-comp}
    Let $(\mcC, \mbE, \mathfrak{s})$ be an extriangulated category. Then, $\mcC$ has an
    $\mbE$-compatible triangulated structure $(\mcC, [1], \Delta)$ if and only if for
    every object $C\in \mcC$, the morphism $C\to 0$ is an $\mbE$-inflation and the 
    morphism $0\to C$ is an $\mbE$-deflation.
\end{theorem}

The examples shown so far correspond to the case of Frobenius reduction and therefore a triangulated structure is expected. We finish this section with another example of an extriangulated category whose quotient category determined by the class of its projective-injective objects is non $\mbE$-compatible 
for any triangulated structure on it. 

\begin{example}\label{ex: red no frobenius}
Let $\Lambda$ be the self-injective Nakayama algebra given in Example~\ref{ex4}.

We depict the Auslander-Reiten quiver of
$\underline{\modu}(\Lambda)$ in the diagram below as follows:
\[
\adjustbox{scale=0.64,center}{
\begin{tikzcd}[column sep=tiny,row sep=huge]
	\times && a && \bullet && \blacklozenge && \bullet && \bullet && \bullet && \spadesuit && \times && \times && \times && \times && \times \\
	{{}} & a && \bullet && \bullet && \bullet && \bullet && \bullet && \bullet && b && \times && \times && \times && \times & {{}} \\
	\clubsuit && \blacklozenge && \bullet && \bullet && \bullet && \blacklozenge && \bullet && \bullet && b && \times && \bigstar && \times && \clubsuit
	\arrow[dashed, no head, from=2-1, to=2-2]
	\arrow[dashed, no head, from=2-24, to=2-25]
	\arrow[dashed, no head, from=1-1, to=1-3]
	\arrow[dashed, no head, from=1-3, to=1-5]
	\arrow[dashed, no head, from=1-5, to=1-7]
	\arrow[dashed, no head, from=1-7, to=1-9]
	\arrow[dashed, no head, from=1-9, to=1-11]
	\arrow[dashed, no head, from=1-11, to=1-13]
	\arrow[dashed, no head, from=1-13, to=1-15]
	\arrow[dashed, no head, from=1-15, to=1-17]
	\arrow[dashed, no head, from=1-17, to=1-19]
	\arrow[dashed, no head, from=1-19, to=1-21]
	\arrow[dashed, no head, from=1-21, to=1-23]
	\arrow[dashed, no head, from=1-23, to=1-25]
	\arrow[dashed, no head, from=2-24, to=2-22]
	\arrow[dashed, no head, from=3-25, to=3-23]
	\arrow[dashed, no head, from=2-22, to=2-20]
	\arrow[dashed, no head, from=3-23, to=3-21]
	\arrow[dashed, no head, from=2-20, to=2-18]
	\arrow[dashed, no head, from=3-21, to=3-19]
	\arrow[dashed, no head, from=2-18, to=2-16]
	\arrow[dashed, no head, from=3-19, to=3-17]
	\arrow[dashed, no head, from=2-16, to=2-14]
	\arrow[dashed, no head, from=3-17, to=3-15]
	\arrow[dashed, no head, from=3-15, to=3-13]
	\arrow[dashed, no head, from=3-13, to=3-11]
	\arrow[dashed, no head, from=2-14, to=2-12]
	\arrow[dashed, no head, from=2-12, to=2-10]
	\arrow[dashed, no head, from=2-2, to=2-4]
	\arrow[dashed, no head, from=2-4, to=2-6]
	\arrow[dashed, no head, from=2-6, to=2-8]
	\arrow[dashed, no head, from=2-8, to=2-10]
	\arrow[dashed, no head, from=3-1, to=3-3]
	\arrow[dashed, no head, from=3-3, to=3-5]
	\arrow[dashed, no head, from=3-5, to=3-7]
	\arrow[dashed, no head, from=3-7, to=3-9]
	\arrow[dashed, no head, from=3-9, to=3-11]
	\arrow[from=1-1, to=2-2]
	\arrow[from=2-2, to=3-3]
	\arrow[from=3-1, to=2-2]
	\arrow[from=2-2, to=1-3]
	\arrow[from=3-3, to=2-4]
	\arrow[from=2-4, to=1-5]
	\arrow[from=1-3, to=2-4]
	\arrow[from=2-4, to=3-5]
	\arrow[from=3-5, to=2-6]
	\arrow[from=1-5, to=2-6]
	\arrow[from=2-6, to=3-7]
	\arrow[from=2-6, to=1-7]
	\arrow[from=1-7, to=2-8]
	\arrow[from=2-8, to=3-9]
	\arrow[from=3-7, to=2-8]
	\arrow[from=2-8, to=1-9]
	\arrow[from=1-9, to=2-10]
	\arrow[from=2-10, to=3-11]
	\arrow[from=3-9, to=2-10]
	\arrow[from=2-10, to=1-11]
	\arrow[from=1-11, to=2-12]
	\arrow[from=2-12, to=3-13]
	\arrow[from=3-11, to=2-12]
	\arrow[from=2-12, to=1-13]
	\arrow[from=1-13, to=2-14]
	\arrow[from=2-14, to=3-15]
	\arrow[from=3-13, to=2-14]
	\arrow[from=2-14, to=1-15]
	\arrow[from=1-15, to=2-16]
	\arrow[from=3-15, to=2-16]
	\arrow[from=2-16, to=3-17]
	\arrow[from=2-16, to=1-17]
	\arrow[from=1-17, to=2-18]
	\arrow[from=2-18, to=3-19]
	\arrow[from=3-17, to=2-18]
	\arrow[from=2-18, to=1-19]
	\arrow[from=1-19, to=2-20]
	\arrow[from=2-20, to=3-21]
	\arrow[from=3-19, to=2-20]
	\arrow[from=2-20, to=1-21]
	\arrow[from=1-21, to=2-22]
	\arrow[from=2-22, to=3-23]
	\arrow[from=3-21, to=2-22]
	\arrow[from=2-22, to=1-23]
	\arrow[from=1-23, to=2-24]
	\arrow[from=2-24, to=3-25]
	\arrow[from=3-23, to=2-24]
	\arrow[from=2-24, to=1-25]
\end{tikzcd}
}
\]
where the first and the last column are identified.\\

Let $\mcY=\add(a \cup b)$ and $\mcC\subseteq \underline{\modu}(\Lambda)$ be as in Example~\ref{ex4} rewritten as 
$$\mcC=\add (a \cup \clubsuit\cup \bullet\cup \blacklozenge\cup b\cup \spadesuit \cup \bigstar).$$

Under the above setting, we have that:
$\mcY$ is $4$-rigid, functorially finite in $\mcC$ and the equality $$\mcY^{\perp_{\leq 3}}={}^{\perp_{\leq 3}}\mcY=
\add (a \cup \clubsuit\cup \blacklozenge\cup b \cup \spadesuit \cup \bigstar)$$ holds. Thus, the reduction of $\mcC$ at $\mcY$ is $\mathcal{R}^{3}_\mcC(\mcY)=\mcY^{\perp_{\leq 3}}$ which is not Frobenius by 
\cite[Corollary 3.9]{huerta2024reduction}.

Notice also that $\mcC$ is a Krull-Schmidt 
extriangulated category by Proposition~\ref{pro: KS}. So, from 
Lemma~\ref{lem: KS}, the equality 
$J':=\add(\mcY\cup J)=\add(a\cup b \cup \bigstar)$
holds true. Thus, by
Theorem~\ref{teo: equiv extriang} we have 
an equivalence of extriangulated categories
$$\mathcal{R}_\mcC^{3}(\mcY)/[J']\cong \pi(\mcY^{\perp_{\leq 3}})/ [\add(\pi(\mcY))]=\add (\clubsuit\cup \blacklozenge\cup \spadesuit).$$

Finally, notice that the 
morphism $0\to \clubsuit$ is not an $\mbE$-deflation in $\mathcal{R}^{3}_\mcC(\mcY)/[J']$ due to $\clubsuit$ is projective. 
Therefore, the quotient 
$\mathcal{R}^{3}_\mcC(\mcY)/[J']$
does not admit an $\mbE$-compatible 
triangulated structure by Theorem~\ref{teo: E-comp}.
\end{example}

%%%%%%%%%%%%%%%%%%%%%%%%%%%%%%%%%%%%%%%%%%%%%%%%%%%%%%%%%%%%%%%%%%%%%%
%%%%%%%%%%%%%%%%%%%%%%%%%%%%%%%%%%%%%%%%%%%%%%%%%%%%%%%%%%%%%%%%%%%%%%
%%%%%%%%%%%%%%%%%%%%%%%%%%%%%%%%%%%%%%%%%%%%%%%%%%%%%%%%%%%%%%%%%%%%%%
%%%%%%%%%%%%%%%%%%%%%%%%%%%%%%%%%%%%%%%%%%%%%%%%%%%%%%%%%%%%%%%%%%%%%%
%%%%%%%%%%%%%%%%%%%%%%%%%%%%%%%%%%%%%%%%%%%%%%%%%%%%%%%%%%%%%%%%%%%%%%

\section*{\textbf{Acknowledgements}}
The author thanks professor Octavio Mendoza for his
comments on the first version of this manuscript.
The author would like to thank professor Corina Sáenz for all her support over the years.

%%%%%%%%%%%%%%%%%%%%%%%%%%%%%%%%%%%%%
%%%%%%%%%%%%%%%%%%%%%%%%%%%%%%%%%%%%%
%%%%%%%%%%%%%%%%%%%%%%%%%%%%%%%%%%%%%
%%%%%%%%%%%%%%%%%%%%%%%%%%%%%%%%%%%%%

%%%%%%%%%%%%%%%%%%%%%%%%%%%%%%%%%%%%%
%%%%%%%%%%%%%%%%%%%%%%%%%%%%%%%%%%%%%
%%%%%%%%%%%%%%%%%%%%%%%%%%%%%%%%%%%%%
%%%%%%%%%%%%%%%%%%%%%%%%%%%%%%%%%%%%%

\bibliographystyle{plain}
\bibliography{biblio23}

\begin{thebibliography}{10}

\bibitem{aihara2012silting}
T.~Aihara and O.~Iyama.
\newblock Silting mutation in triangulated categories.
\newblock {\em Journal of the London Mathematical Society}, 85(3):633--668,
  2012.

\bibitem{bennetttransport}
R.~Bennett-Tennenhaus and A.~Shah.
\newblock Transport of structure in higher homological algebra.
\newblock {\em Journal of Algebra}, 574:514--549, 2021.

\bibitem{CH25}
J.~C. Cala and S.~R. Hern\'andez.
\newblock A characterization of closed subfunctors through $3\times 3$-lemma
  property in extriangulated categories.
\newblock {\em Preprint. arXiv:2504.15579}, 2025.

\bibitem{FMP23}
E.~Faber, B.~R. Marsh, and M.~Pressland.
\newblock Reduction of {F}robenius extriangulated categories.
\newblock {\em arXiv:2308.16232}, 2023.

\bibitem{huerta2024reduction}
M.~Y. Huerta.
\newblock Reduction for $(n+ 2) $-rigid subcategories in extriangulated
  categories.
\newblock {\em Preprint arXiv:2401.17516}, 2024.

\bibitem{HMSS2023quotient}
M.~Y. Huerta, O.~Mendoza, C.~S{\'a}enz, and V.~Santiago.
\newblock Quotient categories with exact structure from $(n+ 2) $-rigid
  subcategories in extriangulated categories.
\newblock {\em Preprint arXiv:2309.14576}, 2023.

\bibitem{IYmutation}
O.~Iyama and Y.~Yoshino.
\newblock Mutation in triangulated categories and rigid {C}ohen-{M}acaulay
  modules.
\newblock {\em Invent. Math.}, 172(1):117--168, 2008.

\bibitem{jasso2015reduction}
G.~Jasso.
\newblock Reduction of $\tau$-tilting modules and torsion pairs.
\newblock {\em International Mathematics Research Notices},
  2015(16):7190--7237, 2015.

\bibitem{HK15}
H.~Krause.
\newblock Krull-schmidt categories and projective covers.
\newblock {\em Expo. Math.}, (33):535--549, 2015.

\bibitem{LNheartsoftwin}
Y.~Liu and H.~Nakaoka.
\newblock Hearts of twin cotorsion pairs on extriangulated categories.
\newblock {\em Journal of Algebra}, 528:96--149, 2019.

\bibitem{LZ20arising}
Y.~Liu and P.~Zhou.
\newblock Abelian categories arising from cluster tilting subcategories.
\newblock {\em Appl. Categ. Structures}, 28(4):575--594, 2020.

\bibitem{MDZtheoryAB}
Y.~Ma, N.~Ding, and Y.~Zhang.
\newblock Auslander-{B}uchweitz {A}pproximation {T}heory for {E}xtriangulated
  {C}ategories.
\newblock {\em Preprint. arXiv:2006.05112}, 2020.

\bibitem{msapato2024characterization}
D.~Msapato.
\newblock A characterization of extriangulated categories with triangulated
  structure.
\newblock {\em Communications in Algebra}, pages 1--16, 2024.

\bibitem{NOSlocalization}
H.~Nakaoka, Y.~Ogawa, and A.~Sakai.
\newblock Localization of extriangulated categories.
\newblock {\em Journal of Algebra}, 611:341--398, 2022.

\bibitem{NP19}
H.~Nakaoka and Y.~Palu.
\newblock Extriangulated categories, {H}ovey twin cotorsion pairs and model
  structures.
\newblock {\em Cah. Topol. G{\'e}om. Diff{\'e}r. Cat{\'e}g}, 60(2):117--193,
  2019.

\bibitem{Ogawa}
Y.~Ogawa.
\newblock Auslander's defects over extriangulated categories: An application
  for the general heart construction.
\newblock {\em Journal of the Mathematical Society of Japan}, 73(4):1063--1089,
  2021.

\end{thebibliography}
\end{document}